\documentclass[11pt]{amsart}
\usepackage{amssymb,amsthm,amsmath,amstext,tikz-cd}
\usepackage{mathrsfs}  % allows nice script font for sheaves
\usepackage{bm}        % allows bold italic letters in math mode
\usepackage{mathtools} % allows more extendible arrows
\usepackage{changepage,enumitem}
\usepackage{color,fullpage}

\usepackage[all]{xy}
 \DeclareFontFamily{U}{wncy}{}
    \DeclareFontShape{U}{wncy}{m}{n}{<->wncyr10}{}
    \DeclareSymbolFont{mcy}{U}{wncy}{m}{n}
    \DeclareMathSymbol{\Sha}{\mathord}{mcy}{"58} 

\theoremstyle{plain}
\newtheorem{theorem}{Theorem}[section]
\newtheorem{prop}[theorem]{Proposition}
\newtheorem{lemma}[theorem]{Lemma}
\newtheorem{cor}[theorem]{Corollary}

\newtheorem*{theoremintro*}{Theorem}

\newtheorem*{propintro*}{Proposition}

\theoremstyle{definition}
\newtheorem{defin}[theorem]{Definition}

\newtheorem{example}[theorem]{Example}
\newtheorem{examples}[theorem]{Examples}
\newtheorem{notation}[theorem]{Notation}

\theoremstyle{remark}
\newtheorem{remark}[theorem]{Remark}

\newtheorem*{ack}{Acknowledgments}

\newcommand{\au}{\text{AU}}

\newcommand{\C}{\mathbb C}
\newcommand{\F}{\mathbb F}
\newcommand{\R}{\mathbb R}
\renewcommand{\P}{\mathbb P}
\newcommand{\Q}{\mathbb Q}

\newcommand{\Char}{\text{char }}

\newcommand{\layer}{\text{layer}}

\usepackage[backref=page]{hyperref}
\hypersetup{pdftitle={Universal quadratic forms}}
\hypersetup{pdfauthor={Connor Cassady}}
\hypersetup{colorlinks=true,linkcolor=red,anchorcolor=blue,citecolor=blue}

\title[Universal quadratic forms]{Universal quadratic forms over semi-global fields}
\date{\today}
\author[Cassady]{Connor Cassady}
\thanks{
The author was supported in part by NSF grants DMS-1805439 and DMS-2102987.
}

\begin{document}
\begin{abstract}
We study anisotropic universal quadratic forms over semi-global fields; i.e., over one-variable function fields over complete discretely valued fields. In particular, given a semi-global field $F$, we compute both the $m$-invariant of $F$ and the set of dimensions of anisotropic universal quadratic forms over $F$. We also define the strong $m$-invariant of a field $k$ and show that it behaves analogously to the strong $u$-invariant of $k$, defined by Harbater, Hartmann, and Krashen. Our main tool in this study is the local-global principle for isotropy of quadratic forms over a semi-global field with respect to particular sets of overfields.
\end{abstract}
\maketitle
\vspace{-1cm}

\section*{Introduction}
A quadratic form $q$ over a field $k$ is called \textit{universal} if $q$ represents all non-zero elements of~$k$; i.e., for all $a \in k^{\times}$ there is some $x$ such that $q(x) = a$. The question of determining which quadratic forms are universal goes back hundreds of years. For example, a universality result is given by Lagrange's Four Square Theorem of 1770: all positive integers can be written as a sum of four squares. In other words, the quadratic form $x_1^2 + x_2^2 + x_3^2 + x_4^2$ over $\mathbb{Q}$ represents all positive integers. Over the rational numbers, the question of universality is typically restricted to asking if a quadratic form represents all positive integers, and has led to results like the ``15-Theorem'' of Conway-Schneeberger \cite{bhargava, conway}, the ``290-Theorem'' of Bhargava-Hanke \cite{bh05}, and the ``451-Theorem'' of Rouse \cite{rouse}. These questions have also been asked over various totally real number fields in, e.g., \cite{bk18, ky21}. What these fields have in common is that they all have infinite $u$-invariant. Recall that the $u$-\textit{invariant} of a field $k$, $u(k)$, is defined to be the maximal dimension of an anisotropic quadratic form over $k$, and if no such maximum exists we say $u(k) = \infty$ \cite[Definition~XI.6.1]{lam}. In this paper, we will be interested in studying universal quadratic forms over fields with finite $u$-invariant.

Any regular isotropic quadratic form over a field $k$ of characteristic $\ne 2$ is universal \cite[Theorem~I.3.4(3)]{lam}, so we will focus on studying anisotropic universal quadratic forms over $k$. We are particularly interested in studying the $m$-\textit{invariant} of~$k$, $m(k)$, defined in \cite{m-inv} to be the minimal dimension of an anisotropic universal quadratic form over $k$, and in studying the set~$\au(k)$ of dimensions of anisotropic universal quadratic forms over $k$. Most of our investigation is focused on these problems over a \textit{semi-global field}; i.e., over a one-variable function field over a complete discretely valued field. Such fields have been extensively studied in, e.g., \cite{cps12, hh10, weier, hhk15}. Our main approach to studying anisotropic universal quadratic forms over semi-global fields is to use the local-global principle for isotropy with respect to various sets of overfields. This approach has been used previously to compute the $u$-invariants of semi-global fields (see, e.g., \cite{cps12, hhk09, refine}).

\textbf{Main results and structure.} The structure of this manuscript is as follows. After establishing notation and discussing background in Section \ref{notation section}, we prove preliminary results about anisotropic universal quadratic forms in Section \ref{preliminaries}. 

In Section \ref{red graph and m=2} we use a combinatorial object called the \textit{reduction graph} associated to a semi-global field $F$ (defined in \cite[Section~6]{hhk15}) to give necessary and sufficient conditions for $m(F) = 2$. The main result of this section is the following (see Section \ref{red graph and m=2} for terminology).
\begin{propintro*}[\ref{m=2 iff loop}]
Let $T$ be a complete discrete valuation ring of residue characteristic $\ne 2$ with fraction field $K$. Let $F$ be a one-variable function field over $K$. Then $m(F) \geq 2$, and $m(F) = 2$ if and only if there is a regular model of $F$ whose reduction graph is not a tree.
\end{propintro*}

Inspired by the strong $u$-invariant of a field $k$, $u_s(k)$, defined in \cite{hhk09}, in Section \ref{strong m} we define the \textit{strong $m$-invariant} of $k$, $m_s(k)$, (see Definition \ref{strong m-inv}) and prove
\begin{theoremintro*}[\ref{strong m of cdvf}]
Let $K$ be a complete discretely valued field whose residue field $k$ has characteristic~$\ne 2$. If $m_s(k) = u_s(k)$, then $m_s(K) = 2m_s(k)$.
\end{theoremintro*}
Theorem \ref{strong m of cdvf} is analogous to \cite[Theorem~4.10]{hhk09}, which states that if $K$ is a complete discretely valued field with residue field $k$ of characteristic $\ne 2$, then $u_s(K) = 2u_s(k)$. 

In Section \ref{all universal forms} we study the set $\au(F)$ of dimensions of anisotropic universal quadratic forms over a semi-global field $F$. The main result of this section is the following (see Section \ref{all universal forms} for notation and terminology).
\begin{theoremintro*}[\ref{au equals union}]
Let $T$ be a complete discrete valuation ring with residue field $k$ of characteristic $\ne 2$ such that $m_s(k) = u_s(k) < \infty$. Let $\mathscr{X}$ be a regular connected projective $T$-curve whose closed fiber $X$ has irreducible components $X_1, \ldots, X_s$. For $1 \leq i \leq s$ let $\eta_i$ be the generic point of~$X_i$, let~$\Gamma$ be the reduction graph of $\mathscr{X}$, and let $F$ be the function field of $\mathscr{X}$. Then
\[
	\emph{AU}(F) = \begin{cases}
		\{2\} \cup \bigcup_{i = 1}^s \left\{r_1 + r_2 \mid r_1, r_2 \in \emph{AU}(\kappa(\eta_i))\right\} &\text{ if $\Gamma$ is not a tree}, \\
		 \bigcup_{i = 1}^s \left\{r_1 + r_2 \mid r_1, r_2 \in \emph{AU}(\kappa(\eta_i))\right\} &\text{ if $\Gamma$ is a tree}.
	\end{cases}
\]
\end{theoremintro*}

Finally, in Section \ref{classifying curves} we study anisotropic universal quadratic forms over semi-global fields $F$ over an $n$-local field over a field $k$. In particular, we define the \textit{layer} of $F$ (see Definition \ref{layer defin}) in order to compute $m(F)$. For example, we prove the following special cases of Propositions \ref{m-inv of fully arboreal} and~\ref{m-inv if layer = j} (see Section \ref{classifying curves} for notation and terminology). The layer of $F$ depends only on the geometry of a model of $F$, so Propositions \ref{m-inv of fully arboreal} and \ref{m-inv if layer = j} link the arithmetic of $F$ to its geometry.
\begin{propintro*}
Let $k$ be a field of characteristic $\ne 2$ that is either algebraically closed or finite. For any integer $n \geq 1$ let $K$ be an $n$-local field over $k$, and let $F$ be a semi-global field over $K$. If $F$ is not fully arboreal, then $m(F) = 2^{\emph{layer}(F)}$. If $F$ is fully arboreal, then $m(F) = 2^{n+\varepsilon}$, where $\varepsilon = 1$ if~$k$ is algebraically closed and $\varepsilon = 2$ if $k$ is finite.
\end{propintro*}

\section{Background and Notation}
\label{notation section}
All of the fields considered will have characteristic different from 2, and all quadratic forms (occasionally referred to just as forms) considered will be nondegenerate (or \textit{regular}). Our notation and terminology follows \cite{lam}, and we assume familiarity with basic notions of quadratic form theory (see \cite[Chapter I]{lam}).

Let $q$ be an $n$-dimensional quadratic form over a field $k$. Because the field $k$ has characteristic~$\ne 2$, we can diagonalize $q$ over $k$ and write $q \simeq \langle a_1, \ldots, a_n \rangle$ with each $a_i \in k^{\times}$. Here the symbol $\simeq$ denotes isometry of quadratic forms over $k$. We call the elements $a_1, \ldots, a_n \in k^{\times}$ the \textit{entries} of $q$. For any $a \in k^{\times}$, we let $a \cdot q$ denote the quadratic form $\langle a \rangle \otimes q$. The set of elements of $k^{\times}$ represented by $q$ over $k$ will be denoted by $D_k(q)$. Throughout this manuscript we will repeatedly use two elementary results that give necessary and sufficient conditions for an element $a \in k^{\times}$ to belong to $D_k(q)$. The first is the Representation Criterion \cite[Theorem~I.2.3]{lam}, which states that $a \in D_k(q)$ if and only if there is a quadratic form $q'$ over~$k$ such that $q \simeq \langle a \rangle \perp q'$. The second is the First Representation Theorem \cite[Corollary~I.3.5]{lam}, which states that $a \in D_k(q)$ if and only if $q \perp \langle -a \rangle$ is isotropic. The First Representation Theorem is particularly useful since it allows us to use results about isotropy of quadratic forms to study universality.

A valuable tool when studying isotropy of quadratic forms over a field $k$ is the local-global principle for isotropy with respect to some set of overfields. Given a quadratic form $q$ over~$k$ and a collection $\{k_i\}_{i \in I}$ of overfields of $k$ (i.e., $k \subseteq k_i$ for all $i \in I$), we say that $q$ \textit{satisfies the local-global principle for isotropy with respect to $\{k_i\}_{i \in I}$} if $q$ being isotropic over $k_i$ for all $i \in I$ implies that~$q$ is isotropic over $k$. Whether the local-global principle for isotropy holds with respect to various sets of overfields has been extensively studied (see, e.g., \cite{cas23, cps12, hhk15, hu12, gup}). 

In this paper, we will encounter several different collections of overfields, one of which being the set of $v$-adic completions $k_v$ of $k$ with respect to various non-trivial discrete valuations $v$ on $k$. In this context, we will frequently use Springer's Theorem on quadratic forms over complete discretely valued fields \cite[Proposition~VI.1.9]{lam}, which we will refer to as Springer's Theorem. Springer's Theorem states that, over a complete discretely valued field $K$ with uniformizer $\pi$, valuation ring~$\mathcal{O}$, and residue field $\kappa$ with $\Char \kappa \ne 2$, a quadratic form $q \simeq q_1 \perp \pi \cdot q_2$ (where the entries of $q_1, q_2$ are all units in $\mathcal{O})$ is anisotropic over $K$ if and only if both residue forms $\overline{q}_1, \overline{q}_2$ are anisotropic over~$\kappa$. 

\section{Preliminaries}
\label{preliminaries}
Let $k$ be a field of characteristic $\ne 2$ with finite $u$-invariant. Then any quadratic form $q$ over $k$ with $\dim q = u(k)$ must be universal over $k$, and any quadratic form over $k$ of dimension $> u(k)$ is isotropic. So if $u(k) < \infty$ we have
\[
	u(k) = \max\{\dim q \geq 1 \mid q \text{ an anisotropic universal quadratic form over $k$}\}.
\]
This then naturally leads to the problem of determining the minimal dimension of an anisotropic universal quadratic form over $k$, which is precisely the definition of the $m$-invariant of $k$.
\begin{defin}[\cite{m-inv}]
\label{m-inv}
Let $k$ be a field. The \textit{$m$-invariant} of $k$ is defined by
\[
	m(k) = \min\{\dim q \geq 1 \mid q \text{ an anisotropic universal quadratic form over $k$}\}.
\]
If there are no anisotropic universal quadratic forms over $k$ we say that $m(k) = \infty$.
\end{defin}
\begin{examples}
\begin{enumerate}[label=(\arabic*)]
	\item $m(\R) = \infty$,
	
	\item $m(\C) = 1$,
	
	\item $m(\F_p) = 2$ for any odd prime $p$.
\end{enumerate}
\end{examples}
We now record some preliminary results about the $m$-invariant. Immediately from its definition we see that $1 \leq m(k) \leq u(k)$ for any field $k$ of characteristic $\ne 2$. Moreover, there are certain integers that can never be the $m$-invariant of a field $k$. Indeed, for any field $k$ of characteristic $\ne 2$, $m(k) \ne 3, 5$ \cite[1.1a),~p. 194]{m-inv}, and if $k$ is a \textit{linked} field (see, e.g, \cite[p.~370]{lam} for the definition of a linked field), then $m(k) \ne 7$ \cite[1.1b),~p. 195]{m-inv}. These statements are similar to \cite[Proposition~XI.6.8]{lam}, which states that, for any field $k$ with $\Char k \ne 2$, $u(k) \ne 3, 5,$ or $7$.

We can make the statement $m(k) \leq u(k)$ more precise. Indeed, $m(k)$ and $u(k)$ are always separated by a power of 2. 
\begin{prop}
\label{separated by power of 2}
Let $k$ be a field of characteristic $\ne 2$, and let $n \geq 0$ be the largest integer such that $2^n \leq u(k)$. Then $m(k) \leq 2^n \leq u(k)$. 
\end{prop}
\begin{proof}
If $u(k) = 1$, then $n = 0$ and $m(k) = 1 \leq 2^0$ \cite[Proposition~1.3]{m-inv}. If $u(k) \geq 2$, then this is precisely the statement of \cite[Corollary~1.6]{m-inv}.
\end{proof}
As an immediate consequence of Proposition \ref{separated by power of 2}, we have
\begin{cor}
\label{if m=u}
If $k$ is a field of characteristic $\ne 2$ with $m(k) = u(k) < \infty$, then $m(k) = u(k) = 2^n$ for some integer $n \geq 0$.
\end{cor}
\begin{remark}
In \cite{m-inv} it was observed that determining when $m(k) = 2$ can be related to studying the \textit{Kaplansky radical} of $k$, denoted $R(k)$, defined by Kaplansky in \cite{kap}. By \cite[Proposition~XII.6.1]{lam}, $R(k)$ consists of the elements $a \in k^{\times}$ such that $\langle 1, -a \rangle$ is universal over $k$. By definition, $R(k) \subseteq k^{\times}$, and because isotropic quadratic forms are universal, we have $k^{\times 2} \subseteq R(k)$. By \cite[1.2, p. 195]{m-inv}, $m(k) = 2$ if and only if $k^{\times 2} \subset R(k) \subset k^{\times}$, where both inclusions are strict. The Kaplansky radical was also studied in, e.g., \cite{becher-leep}, where Becher and Leep called a field $k$ \textit{radical-free} if $R(k) = k^{\times 2}$. If $k^{\times} \ne k^{\times 2}$, then $k$ is radical-free if and only if $m(k) > 2$.
\end{remark}
Much like the $u$-invariant, computing the $m$-invariant of a given field is a challenging problem. For some simple fields like $\C, \F_p$, and $\R$, we know exact values of the $m$-invariant, each of which agrees with the respective $u$-invariant. However, there are fields $k$ with $m(k) < u(k)$. For example, in \cite[Example~2.10]{m-inv}, for any positive integer $n \geq 2$, Gesqui\`{e}re and Van Geel constructed a field $k$ with $m(k) = 2$ and $u(k) = 2n$, and in \cite[Proposition~4.3]{hoff} Hoffmann constructed a field~$k$ with $m(k) = 6$, which by Proposition \ref{separated by power of 2} must have $m(k) < u(k)$. We will see other examples of such fields as we continue (see, e.g., Remark \ref{sg fields with m<u}).

For fields $k$ with $m(k) = u(k) < \infty$, the only possible dimension of an anisotropic universal quadratic form over $k$ is this common value $m(k) = u(k)$. However, if $m(k) < u(k)$, then there may be other such dimensions. With this in mind, we define the set $\au(k)$ by
\[
	\au(k) = \left\{\dim q \mid q \text{ an anisotropic universal quadratic form over $k$}\right\}.
\]
From the above observations, for a field $k$ with $u(k) < \infty$, we see that $\au(k)$ is non-empty, with $m(k), u(k) \in \au(k)$.

Throughout this manuscript, we will use the local-global principle for isotropy with respect to various sets of overfields to compute both lower bounds and exact values of the $m$-invariant for certain fields $k$, as well as to compute the set $\au(k)$. As a part of this process, it is beneficial to collect several results about universal quadratic forms over complete discretely valued fields.
\begin{lemma}
\label{anisotropic universal residue forms}
Let $K$ be a complete discretely valued field with residue field $k$ of characteristic $\ne 2$. Let $q$ be a quadratic form over $K$, and for a uniformizer $\pi$ of $K$, write $q \simeq q_1 \perp \pi \cdot q_2$, where the entries of $q_1$ and $q_2$ are units in the valuation ring $\mathcal{O}_K$ of $K$. Then $q$ is anisotropic and universal over~$K$ if and only if both residue forms $\overline{q}_1, \overline{q}_2$ have positive dimension and are anisotropic and universal over $k$.
\end{lemma}
\begin{proof}
First assume that both residue forms $\overline{q}_1$ and $\overline{q}_2$ are anisotropic and universal over $k$ with $\dim \overline{q}_1, \dim \overline{q}_2 > 0$. Then $q$ is anisotropic over $K$ by Springer's Theorem. Moreover, since $\overline{q}_1$ and~$\overline{q}_2$ are universal over $k$, the First Representation Theorem and Springer's Theorem imply that the forms $q \perp \langle -u \rangle$ and $q \perp \langle -\pi u\rangle$ are isotropic over $K$ for all $u \in \mathcal{O}_K^{\times}$. Given any $a \in K^{\times}$, we conclude that $q \perp \langle -a \rangle$ is isotropic over $K$ since there is some $u \in \mathcal{O}_K^{\times}$ (depending on $a$) such that $q \perp \langle -a \rangle$ is isometric to either $q \perp \langle -u \rangle$ or $q \perp \langle -\pi u \rangle$. By the First Representation Theorem we conclude that~$q$ is universal over $K$.

Conversely, suppose $q$ is anisotropic and universal over $K$. By Springer's Theorem, both residue forms $\overline{q}_1, \overline{q}_2$ must be anisotropic over $k$, and we now show that $\dim \overline{q}_1, \dim \overline{q}_2 > 0$. By contradiction, suppose not and, without loss of generality, assume $\dim \overline{q}_1 = 0$, i.e., $q \simeq \pi \cdot q_2$. Then by Springer's Theorem the form $q \perp \langle -u \rangle$ is anisotropic over $K$ for all $u \in \mathcal{O}_K^{\times}$. This contradicts our assumption that $q$ is universal over $K$, hence $\dim \overline{q}_1, \dim \overline{q}_2 > 0$. Finally, we show that both residue forms are universal over $k$. Let $\overline{u} \in k^{\times}$ be arbitrary and let $u \in \mathcal{O}_K^{\times}$ be a unit lift of~$\overline{u}$ to~$K$. Since~$q$ is universal over $K$, both $q \perp \langle -u \rangle$ and $q \perp \langle -\pi u \rangle$ are isotropic over $K$ by the First Representation Theorem. So by Springer's Theorem, $\overline{q}_1 \perp \langle -\overline{u} \rangle$ and $\overline{q}_2 \perp \langle -\overline{u} \rangle$ are isotropic over~$k$. Since $\overline{u} \in k^{\times}$ was arbitrary, the First Representation Theorem implies that both $\overline{q}_1$ and~$\overline{q}_2$ are universal over~$k$.
\end{proof}

As an immediate corollary of Lemma \ref{anisotropic universal residue forms}, for a complete discretely valued field $K$ with residue field $k$, the set $\au(K)$ is completely determined by $\au(k)$.
\begin{cor}
\label{au of cdvf}
Let $K$ be a complete discretely valued field with residue field $k$ of characteristic~$\ne 2$. Then $\emph{AU}(K) = \left\{r_1 + r_2 \mid r_1, r_2 \in \emph{AU}(k) \right\}$.
\end{cor}
\begin{proof}
Let $q$ be any anisotropic universal quadratic form over $K$, and write $q \simeq q_1 \perp \pi \cdot q_2$, where~$\pi$ is a uniformizer of $K$ and all entries of $q_1, q_2$ are units in the valuation ring of $K$. Then by Lemma~\ref{anisotropic universal residue forms}, both residue forms $\overline{q}_1, \overline{q}_2$ must be anisotropic and universal over $k$. Therefore $\dim \overline{q}_1, \dim \overline{q}_2 \in \au(k)$. Since $q$ was arbitrary, this shows that $\au(K) \subseteq \left\{r_1 + r_2 \mid r_1, r_2 \in \au(k) \right\}$.

To show the reverse containment, let $r_1, r_2 \in \au(k)$ be given, and let $\overline{\varphi}_1, \overline{\varphi}_2$ be anisotropic universal quadratic forms over $k$ of dimensions $r_1, r_2$, respectively. Then for lifts $\varphi_1, \varphi_2$ of $\overline{\varphi}_1, \overline{\varphi}_2$ to $K$, the $(r_1 + r_2)$-dimensional quadratic form $\varphi$ defined over $K$ by $\varphi = \varphi_1 \perp \pi \cdot \varphi_2$ is anisotropic and universal over~$K$ by Lemma \ref{anisotropic universal residue forms}. This completes the proof.
\end{proof}

Corollary \ref{au of cdvf} then implies the following result (see also \cite[Lemma~2.1]{m-inv}).
\begin{cor}
\label{m-inv of cdvf}
Let $K$ be a complete discretely valued field with residue field $k$ of characteristic~$\ne 2$. Then $m(K) = 2m(k)$.
\end{cor}
\begin{proof}
This follows immediately from Corollary \ref{au of cdvf} since the $m$-invariant of a field $F$ is the smallest element of $\au(F)$.
\end{proof}
\begin{lemma}
\label{globally universal implies locally universal}
Let $k$ be any field of characteristic $\ne 2$, and let $v$ be any non-trivial discrete valuation on $k$. If a quadratic form $q$ over $k$ is universal over $k$, then $q$ is universal over its completion $k_v$.
\end{lemma}
\begin{proof}
To show that $q$ is universal over $k_v$, it suffices to show that $q \perp \langle -b_v \rangle$ is isotropic over~$k_v$ for all $b_v \in k_v^{\times}$. To that end, let $b_v \in k_v^{\times}$ be arbitrary. Because $k$ is dense in $k_v$, by an application of Krasner's Lemma \cite[Proposition~7.61]{milne} we can find an element $b \in k^{\times}$ close enough to $b_v$ in~$k_v$ to ensure that $b$ and $b_v$ belong to the same square class of $k_v$. That is, in $k_v$, $b / b_v$ is a square.

Since $q$ is universal over $k$, the form $q \perp \langle -b \rangle$ is isotropic over $k$. Then over $k_v$, we have $q \perp \langle -b_v \rangle \simeq q \perp \langle -b \rangle$. Thus $q \perp \langle -b_v \rangle$ is isotropic over $k_v$, proving the lemma.
\end{proof}

To conclude this section, we prove a lemma that shows how the local-global principle for isotropy can be used to study the $m$-invariant. This lemma will be used throughout the manuscript.
\begin{lemma}
\label{everywhere locally universal implies globally universal}
Let $k$ be a field of characteristic $\ne 2$, and let $\{k_i\}_{i \in I}$ be a set of overfields of $k$. Suppose there exists an integer $n \geq 1$ such that all $(n+1)$-dimensional quadratic forms over $k$ satisfy the local-global principle for isotropy with respect to $\{k_i\}_{i \in I}$. Let $q$ be an $n$-dimensional quadratic form over $k$ that is universal (e.g., isotropic) over $k_i$ for all $i \in I$. Then $q$ is universal over $k$. If in addition $q$ is anisotropic over $k$, then $m(k) \leq n$.
\end{lemma}
\begin{proof}
The second statement follows immediately from the first by the definition of $m(k)$, so it suffices to prove the first statement. Let $a \in k^{\times}$ be arbitrary and consider the $(n+1)$-dimensional quadratic form $q \perp \langle -a \rangle$ over $k$. Since $q$ is universal over $k_i$ for all $i \in I$, it follows that $q \perp \langle -a \rangle$ is isotropic over $k_i$ for all $i \in I$. This, by assumption, implies that $q \perp \langle -a \rangle$ is isotropic over $k$. Since $a \in k^{\times}$ was arbitrary, $q$ must be universal over $k$.
\end{proof}

\section{The reduction graph and the $m$-invariant}
\label{red graph and m=2}
In this section, we will be particularly interested in universal quadratic forms over \textit{semi-global fields}. Recall that a semi-global field is a one-variable function field $F$ over a complete discretely valued field $K$. Such fields have been extensively studied in, e.g., \cite{cps12, hh10, weier, hhk15}. For such fields $F$, the local-global principle for isotropy with respect to particular sets of overfields has been used to calculate $u(F)$ (see, e.g., \cite{cps12, hhk09, refine}). We will use the local-global principle for isotropy to study anisotropic universal quadratic forms over semi-global fields.

We begin by giving a lower bound on the $m$-invariant of any finitely generated field extension of transcendence degree one. The proof is similar to that of \cite[Proposition~3.4]{becher-leep}.
\begin{lemma}
\label{m at least 2 for fg trdeg 1}
Let $k$ be any field of characteristic $\ne 2$, and let $L$ be any finitely generated field extension of transcendence degree one over $k$. Then $m(L) \geq 2$.
\end{lemma}
\begin{proof}
By \cite[Proposition~1.3]{m-inv}, $m(L) = 1$ if and only if $u(L) = 1$, which holds if and only if~$L$ is quadratically closed, i.e., every non-zero element of $L$ is a square. Therefore, to prove the lemma, it suffices to show that $L$ is not quadratically closed. 

First consider the field $k(t)\left(\sqrt{-1}\right) \cong k\left(\sqrt{-1}\right)(t)$. This field is not quadratically closed since~$t$ is not a square in $k(t)\left(\sqrt{-1}\right)$. Now, since $L$ is a finitely generated extension of transcendence degree one over $k$, we can write $L$ as a finite extension of $k(t)$. If $L$ is quadratically closed, then $k(t)\left(\sqrt{-1}\right) \subseteq L$. By \cite[Corollary~VIII.5.11]{lam}, $L$ being quadratically closed implies that $k(t)\left(\sqrt{-1}\right)$ is quadratically closed, which we just showed is not the case. So $L$ is not quadratically closed, completing the proof.
\end{proof}
As an immediate consequence of Lemma \ref{m at least 2 for fg trdeg 1}, we have
\begin{cor}
\label{m at least 2 for sg field}
Let $F$ be a semi-global field over a complete discretely valued field $K$ of residue characteristic $\ne 2$. Then $m(F) \geq 2$.
\end{cor}
We now recall notation and terminology from the patching framework developed by Harbater and Hartmann (see \cite[Section~6]{hh10}).
\begin{notation}
\label{sg fields notation}
Let $T$ be a complete discrete valuation ring with uniformizer $t$, residue field $k$, and fraction field $K$. Let $F$ be a one-variable function field over $K$, and let $\mathscr{X}$ be a \textit{normal model} of~$F$ over $T$, i.e., a normal connected projective $T$-curve with function field $F$. Let $X$ denote the closed fiber of~$\mathscr{X}$. For each point $P \in X$ (not necessarily closed), let $R_P$ denote the local ring of $\mathscr{X}$ at $P$, let $\widehat{R}_P$ denote the completion of $R_P$ with respect to its maximal ideal, and let $F_P$ be the fraction field of $\widehat{R}_P$. A \textit{branch} of $X$ at a closed point $P$ is a height one prime $\wp$ of $\widehat{R}_P$ that contains $t$.
\end{notation}

In fact, in the situation of Notation \ref{sg fields notation}, by \cite{abh, lip}, there are \textit{regular models} $\mathscr{X}$ of~$F$, i.e.,~$\mathscr{X}$ is a normal model of $F$ that is regular as a $T$-curve. In \cite[Section~6]{hhk15}, Harbater, Hartmann, and Krashen defined a bipartite graph associated to the closed fiber $X$ of a regular model~$\mathscr{X}$ of $F$ called the \textit{reduction graph} of $\mathscr{X}$, and showed that the isomorphism class of this graph does not depend on the choice of regular model \cite[Corollary~7.8]{hhk15}. Harbater, Hartmann, and Krashen then showed that, for various algebraic objects over $F$, the validity of certain local-global principles depends on whether the reduction graph of a regular model of $F$ is a tree. The local-global principle we will focus on is the local-global principle for isotropy of quadratic forms.
\begin{theorem}[Harbater-Hartmann-Krashen]
Let $T$ be a complete discrete valuation ring of residue characteristic $\ne 2$ with fraction field $K$. Let $F$ be a one-variable function field over $K$, and let $\mathscr{X}$ be a normal model of $F$ with closed fiber $X$. Let $q$ be a quadratic form of dimension at least $3$ over $F$. Then
\[
	q \text{ is isotropic over $F$ if and only if $q$ is isotropic over $F_P$ for all $P \in X$}.
\]
Moreover, two-dimensional quadratic forms over $F$ satisfy this local-global principle for isotropy if and only if the reduction graph of a regular model of $F$ is a tree.
\end{theorem}
\begin{proof}
See \cite[Theorem~9.3]{hhk15} for the statement about quadratic forms of dimension $\geq 3$, and \cite[Corollary~9.7]{hhk15} for the statement about two-dimensional quadratic forms.
\end{proof}
\begin{remark}
In the situation of Notation \ref{sg fields notation}, by \cite[p. 193]{lip} we can always find a regular model~$\mathscr{X}$ of $F$ whose \textit{reduced} closed fiber $X^{\text{red}}$ is a union of regular curves over $k$, with normal crossings. We can identify the points of the closed fiber $X$ of $\mathscr{X}$ with the points of $X^{\text{red}}$, and the fields $F_P$ are the same if we consider the points $P$ as belonging to $X$ or $X^{\text{red}}$. We will focus more on the reduced closed fiber of a regular model of $F$ in Section \ref{all universal forms}. Furthermore, there is another graph associated to such a model $\mathscr{X}$, namely the \textit{dual graph} (as in \cite[p.~86]{deligne}, \cite[Definition~10.3.17]{liu}). In this normal crossings situation, the reduction graph of Harbater, Hartmann, and Krashen (which is defined for more general regular models) is the barycentric subdivision of the dual graph (see \cite[Remark~6.1(a)]{hhk15}).
\end{remark}
We will now show that we can use the reduction graph of a regular model of a semi-global field~$F$ to determine when $m(F) = 2$.
\begin{lemma}
\label{loop implies m=2}
Let $T$ be a complete discrete valuation ring of residue characteristic $\ne 2$ with fraction field $K$. Let $F$ be a one-variable function field over $K$ with a regular model whose reduction graph is not a tree. Then $m(F) = 2$.
\end{lemma}
\begin{proof}
By Corollary \ref{m at least 2 for sg field} we know that $m(F) \geq 2$. Therefore, to prove the result, it suffices to show that $m(F) \leq 2$. Let $\mathscr{X}$ be a regular model of $F$ whose reduction graph is not a tree, and let $X$ be the closed fiber of $\mathscr{X}$. By \cite[Corollary~9.7]{hhk15} there exists a two-dimensional quadratic form over $F$ that is anisotropic over $F$ but isotropic over $F_P$ for all $P \in X$. However, three-dimensional quadratic forms over $F$ satisfy the local-global principle for isotropy with respect to these overfields \cite[Theorem~9.3]{hhk15}. Thus $m(F) \leq 2$ by Lemma \ref{everywhere locally universal implies globally universal}.
\end{proof}
The next lemma is a rephrasing of Lemma \ref{globally universal implies locally universal} over a semi-global field $F$ in terms of points $P$ on the closed fiber of a particular regular model of $F$. Such a model always exists (see Remark \ref{normal crossings q-models exist}).
\begin{lemma}
\label{universal over F implies universal over F_P}
Let $F$ be a one-variable function field over the fraction field of a complete discrete valuation ring $T$ of residue characteristic $\ne 2$. Let $\mathscr{X}$ be a regular model of $F$ with closed fiber $X$ such that distinct branches at any closed point of $X$ lie on distinct irreducible components of~$X$. Let $q$ be a universal quadratic form over $F$. Then for all points $P \in X$, $q$ is universal over~$F_P$.
\end{lemma}
\begin{proof}
There are two types of points $P \in X$: generic points of irreducible components of $X$, and closed points. We consider these types of points separately.

First, suppose that $P \in X$ is the generic point $\eta$ of an irreducible component of $X$. The codimension one point $\eta$ on $\mathscr{X}$ induces a discrete valuation $v_{\eta}$ on $F$, and in this case the field $F_P = F_{\eta}$ is the completion of $F$ with respect to $v_{\eta}$. Therefore, by Lemma \ref{globally universal implies locally universal}, since $q$ is universal over $F$, it is also universal over $F_P$.

Next, suppose that $P \in X$ is a closed point. To show that $q$ is universal over $F_P$, by the First Representation Theorem, it suffices to show that, for all $b_P \in F_P^{\times}$, the form $q \perp \langle -b_P \rangle$ is isotropic over $F_P$. To that end, let $b_P \in F_P^{\times}$ be arbitrary. Because we assumed distinct branches at $P$ lie on distinct irreducible components of $X$, the conditions of \cite[Corollary~3.3(c)]{weier} are automatically satisfied. Therefore there is some $b \in F^{\times}$ and $c \in F_P^{\times}$ such that $b_P = bc^2$. The form~$q$ is universal over $F$, so the form $q \perp \langle -b \rangle$ is isotropic over $F$. Now over $F_P$, we have
\[
	q \perp \langle -b_P \rangle = q \perp \left\langle -bc^2 \right\rangle \simeq q \perp \langle -b \rangle.
\]
Thus $q \perp \langle -b_P \rangle$ is isotropic over $F_P$, completing the proof.
\end{proof}
As we continue to study anisotropic universal quadratic forms $q$ over a semi-global field $F$, we will want to consider particular regular models of $F$ determined by the quadratic form $q$.
\begin{defin}
\label{normal crossings q-model}
Let $T$ be a complete discrete valuation ring of residue characteristic $\ne 2$ with fraction field $K$. Let $F$ be a one-variable function field over $K$, and let $\mathscr{X}$ be a regular model of~$F$. Let $q = \langle a_1, \ldots, a_n \rangle$ be a regular quadratic form over $F$, and let $D$ be the union of the supports of the divisors of the $a_i$ considered as rational functions on $\mathscr{X}$. We call $\mathscr{X}$ a \textit{normal crossings $q$-model} of $F$ if
\begin{enumerate}[label=\arabic*.]
	\item the singularities of $D$ are normal crossings, and
	
	\item distinct branches at any closed point of the closed fiber $X$ of $\mathscr{X}$ lie on distinct irreducible components of $X$.
\end{enumerate}
\end{defin}
\begin{remark}
\label{normal crossings q-models exist}
Given any regular quadratic form $q$ over a semi-global field $F$, we can always find a normal crossings $q$-model of $F$ by taking a suitable blow-up of a given regular model. Indeed, if we start with a regular model $\mathscr{X}''$ of $F$, then by \cite[Lemma~4.7]{hhk09} we can find a blow-up $\mathscr{X}' \to \mathscr{X}''$ so that condition 1 of Definition \ref{normal crossings q-model} is satisfied. By blowing up $\mathscr{X}'$ at points of intersection of its closed fiber (which does not affect the validity of condition 1), we obtain a regular model~$\mathscr{X}$ of~$F$ satisfying conditions 1 and 2 of Definition \ref{normal crossings q-model}. We will frequently see normal crossings $q$-models arise in this way.
\end{remark}
The next lemma will be useful when considering anisotropic universal quadratic forms over the fields $F_P$ for closed points $P$, and is similar to \cite[Lemma~9.9]{hhk15}.
\begin{lemma}
\label{forms of a prescribed form}
Let $R$ be a regular complete local domain of dimension two, whose residue field~$k$ has characteristic $\ne 2$. Let $E$ be the fraction field of $R$, let $\{x, y\}$ be a generating set of the maximal ideal of $R$, and let $E_y$ be the completion of $E$ with respect to the $y$-adic valuation. Let $q = \langle a_1, \ldots, a_n \rangle$ be a regular quadratic form over $E$ such that, for each $i = 1, \ldots, n$, $a_i = u_ix^{r_i}y^{s_i}$ for some integers $r_i, s_i \geq 0$ and some $u_i \in R^{\times}$. Then
\begin{enumerate}[label=(\alph*)]
	\item $m(E_y) = 4m(k)$.
	
	\item if $q$ is anisotropic and universal over $E$, then it is anisotropic and universal over $E_y$.
\end{enumerate}
\end{lemma}
\begin{proof}
(a) The field $E_y$ is a complete discretely valued field whose residue field $\kappa(y)$ is the fraction field of the complete discrete valuation ring $R / yR$. The field $\kappa(y)$ is therefore also a complete discretely valued field, with residue field $k$. So by applying Corollary \ref{m-inv of cdvf} twice we conclude
	\[
		m(E_y) = 2m(\kappa(y)) = 2(2m(k)) = 4m(k).
	\]
	
(b) Since $q$ is universal over $E$, it is universal over $E_y$ by Lemma \ref{globally universal implies locally universal}. Furthermore, $q$ is of the form in the hypothesis of \cite[Lemma~9.9]{hhk15}. So since $q$ is anisotropic over~$E$, by \cite[Lemma~9.9(a)]{hhk15} it must also be anisotropic over $E_y$.
\end{proof}
We will use the following corollary of Lemma \ref{forms of a prescribed form} at various points of this paper.
\begin{cor}
\label{m of F_P for P closed}
Let $T$ be a complete discrete valuation ring of residue characteristic $\ne 2$ with fraction field $K$. Let $F$ be a one-variable function field over $K$ and let $q$ be a regular quadratic form over $F$. Let $\mathscr{X}$ be a normal crossings $q$-model of $F$ with closed fiber $X$, and let $P \in X$ be any closed point. If $q$ is anisotropic and universal over $F_P$, then $\dim q \geq 4m(\kappa(P))$, where $\kappa(P)$ is the residue field of the ring $\widehat{R}_P$.
\end{cor}
\begin{proof}
The field $F_P$ is the fraction field of the regular complete local domain $\widehat{R}_P$ of dimension two whose residue field $\kappa(P)$ has characteristic $\ne 2$. If $D$ is the union of the supports of the divisors of the entries of $q$ considered as rational functions on $\mathscr{X}$, then because $\mathscr{X}$ is a normal crossings $q$-model of $F$, we can find a local system of parameters $\{x, y\}$ of $\widehat{R}_P$ so that any component of $D$ that passes through $P$ must belong to the zero locus of $xy$ on $\mathscr{X}$. Therefore, after scaling $q$ by an element of the form $x^ry^s$, we may assume that $q$ is of the form in the hypothesis of Lemma~\ref{forms of a prescribed form}. Scaling $q$ by $x^ry^s$ does not affect $q$ being anisotropic and universal over $F_P$, hence $q$ must be anisotropic and universal over $F_{P,y}$ by Lemma \ref{forms of a prescribed form}(b). This implies $\dim q \geq m(F_{P,y})$. Applying Lemma \ref{forms of a prescribed form}(a), we have $\dim q \geq m(F_{P,y}) = 4m(\kappa(P))$.
\end{proof}
We now have the terminology and results needed to show that the converse of Lemma \ref{loop implies m=2} is true.
\begin{lemma}
\label{m=2 implies loop}
Let $T$ be a complete discrete valuation ring of residue characteristic $\ne 2$ with fraction field $K$. Let $F$ be a one-variable function field over $K$. If $m(F) = 2$, then the reduction graph of any regular model of $F$ is not a tree.
\end{lemma}
\begin{proof}
To prove the claim, it suffices to find one regular model of $F$ whose reduction graph is not a tree since the choice of regular model does not affect whether the reduction graph is a tree.

Since $m(F) = 2$, there is an anisotropic universal quadratic form $q$ over $F$ with $\dim q = 2$. Let $\mathscr{X}$ be a normal crossings $q$-model of $F$ with closed fiber $X$. Then by Lemma \ref{universal over F implies universal over F_P}, for all points $P \in X$, since $q$ is universal over $F$, it must also be universal over $F_P$. We will show that $q$ is isotropic over each $F_P$, which implies that $q$ is a two-dimensional counterexample to the local-global principle for isotropy with respect to $\{F_P\}_{P \in X}$. This then proves the claim that the reduction graph of a regular model of $F$ is not a tree \cite[Corollary~9.7]{hhk15}.

We first consider the case that $P \in X$ is a closed point and assume, by contradiction, that $q$ is anisotropic over $F_P$. So $q$ is anisotropic and universal over $F_P$. Applying Corollary \ref{m of F_P for P closed}, we have $\dim q \geq 4m(\kappa(P))$. Since $m(\kappa(P)) \geq 1$, this implies that $\dim q \geq 4$, which is a contradiction since $\dim q = 2$. So $q$ must be isotropic over $F_P$ for all closed points $P \in X$.

Now suppose $P \in X$ is the generic point $\eta$ of an irreducible component of $X$. In this case, the field~$F_{\eta}$ is a complete discretely valued field whose residue field $\kappa(\eta)$ is a finitely generated transcendence degree one extension of the residue field $k$ of $K$. By Lemma \ref{m at least 2 for fg trdeg 1} we have $m(\kappa(\eta)) \geq 2$. Therefore, by Corollary \ref{m-inv of cdvf}, $m(F_{\eta}) = 2m(\kappa(\eta)) \geq 4$. Thus the two-dimensional form $q$ that is universal over~$F_{\eta}$ must be isotropic over $F_{\eta}$. 

We have shown that $q$ is isotropic over $F_P$ for all points $P \in X$, so the proof is complete.
\end{proof}
Combining Corollary \ref{m at least 2 for sg field}, Lemma \ref{loop implies m=2}, and Lemma \ref{m=2 implies loop} together proves the following result.
\begin{prop}
\label{m=2 iff loop}
Let $T$ be a complete discrete valuation ring of residue characteristic $\ne 2$ with fraction field $K$. Let $F$ be a one-variable function field over $K$. Then $m(F) \geq 2$, and $m(F) = 2$ if and only if there is a regular model of $F$ whose reduction graph is not a tree.
\end{prop}
\begin{remark}
\label{sg fields with m<u}
As a consequence of Proposition \ref{m=2 iff loop}, we can find numerous fields $F$ such that $m(F) < u(F)$. Indeed, let $k$ be a finite field of odd characteristic, and for any integer $n \geq 1$ let $K_n = k((t_1))\cdots ((t_n))$ be the field of iterated Laurent series over $k$. Let $F_n$ be a semi-global field over $K_n$ with a regular model whose reduction graph is not a tree (such an $F_n$ always exists \cite[Lemma~5.1]{hkp}). Then by Proposition \ref{m=2 iff loop}, $m(F_n) = 2$, but $u(F_n) = 2^{n+2}$ \cite[Corollary~4.14(b)]{hhk09}.
\end{remark}

\section{The strong $m$-invariant}
\label{strong m}
In \cite{hhk09}, Harbater, Hartmann, and Krashen defined the strong $u$-invariant of a field in order to calculate the $u$-invariant of certain semi-global fields (most notably to show $u(\Q_p(x)) = 8$). In this section, we define the strong $m$-invariant of a field $k$ in order to calculate the $m$-invariant of certain semi-global fields, and show that it relates to the strong $u$-invariant of $k$ in ways that are analogous to how $m(k)$ and $u(k)$ relate to one another. We begin by recalling the definition of the strong $u$-invariant of $k$.
\begin{defin}[\cite{hhk09}]
\label{strong u}
Let $k$ be a field. The \textit{strong $u$-invariant} of $k$, denoted by $u_s(k)$, is the smallest real number $n$ such that
\begin{itemize}
	\item $u(k') \leq n$ for all finite field extensions $k'/k$, and
	
	\item $u(K') \leq 2n$ for all finitely generated field extensions $K'/k$ of transcendence degree one.
\end{itemize}
If these $u$-invariants are arbitrarily large we say that $u_s(k) = \infty$.
\end{defin}
It is noted in \cite{hhk09} that, since the $u$-invariant, if finite, is a positive integer, the strong $u$-invariant always belongs to $\frac{1}{2} \mathbb{N}$. Motivated by this definition, we introduce the following definition.
\begin{defin}
\label{strong m-inv}
Let $k$ be a field. The \textit{strong $m$-invariant} of $k$, denoted by $m_s(k)$, is the largest integer $n$ such that
\begin{itemize}
	\item $m(k') \geq n$ for all finite field extensions $k'/k$, and
	
	\item $m(k'(t)) \geq 2n$ for all finite field extensions $k'/k$.
\end{itemize}
If these $m$-invariants are arbitrarily large we say that $m_s(k) = \infty$.
\end{defin}
\begin{examples}
\label{strong m examples}
\begin{enumerate}[label=(\arabic*)]
	\item Let $k$ be an algebraically closed field of characteristic $\ne 2$. Then $m_s(k) = 1$. Indeed, the only finite extension of $k$ is $k$ itself, and we know that $m(k) = 1$. Moreover, the field $k(t)$ is a $C_1$ field, so $u(k(t)) \leq 2$, and $m(k(t)) \geq 2$ by Lemma \ref{m at least 2 for fg trdeg 1}. Therefore $m(k(t)) = 2$.
	
	\item Let $k$ be a finite field of odd characteristic. Then $m_s(k) = 2$. Indeed, for any finite extension $k'/k$, $m(k') = 2$ since $u(k') = 2$, and $m(k'(t)) = 4$ by, e.g., \cite[Example~2.8]{m-inv}.
	
	\item Let $L$ be any finitely generated transcendence degree one extension of an algebraically closed field $k$ of characteristic $\ne 2$. Then $m_s(L) = u_s(L) = 2$. Indeed, for any finite extension $L'/L$ we have $m(L') \geq 2$ by Lemma \ref{m at least 2 for fg trdeg 1}, and $u(L') \leq 2$ since $L'$ is a $C_1$ field. Furthermore, by \cite[Proposition~3.4]{becher-leep}, $m(L'(t)) \geq 4$, hence $m_s(L) = 2$. For any finitely generated transcendence degree one extension $E$ of $L$, $E$ is a $C_2$ field, so $u(E) \leq 4$. Therefore $u_s(L) = 2$.
\end{enumerate}
\end{examples}

The main goal of this section is to prove Theorem \ref{strong m of cdvf}: for a complete discretely valued field~$K$ with residue field $k$ of characteristic $\ne 2$ satisfying $m_s(k) = u_s(k)$, we have $m_s(K) = 2m_s(k)$. This is an analogue of \cite[Theorem~4.10]{hhk09}, which states that $u_s(K) = 2u_s(k)$. With this goal in mind, we make two observations about the differences between the definitions of $u_s(k)$ and $m_s(k)$.

In the definition of $m_s(k)$, we are asking for the \textit{largest} integer $n$ that gives \textit{lower} bounds on the $m$-invariants of certain field extensions of $k$. This is the opposite of $u_s(k)$, which is asking for the smallest $n$ that gives upper bounds on the $u$-invariants of certain field extensions of $k$. This difference is due to the fact that finding upper bounds for $m(k)$ (respectively lower bounds for $u(k)$) is ``easy'', while finding lower bounds for $m(k)$ (respectively upper bounds for~$u(k)$) is challenging. The next key difference between $m_s(k)$ and $u_s(k)$ is in the second requirement of the definitions. In the definition of $u_s(k)$, we consider all finitely generated field extensions $K'/k$ of transcendence degree one, whereas in the definition of $m_s(k)$, we consider only those transcendence degree one extensions of $k$ of the form $K' = k'(t)$ for some $k'/k$ finite. This is because the $m$-invariant does not behave as uniformly as the $u$-invariant over arbitrary finitely generated field extensions of transcendence degree one, as the following example illustrates.

\begin{example}
\label{Q_p for m_s}
Let $F = \Q_p(x)$ for an odd prime $p$. We will see in Corollary \ref{some m examples} that $m(\Q_p(x)) = 8$. Consider the quadratic extension of $F$ given by $F' = \Q_p(x)\left(\sqrt{x(1-x)(x-p)}\right)$. In \cite[Appendix]{cps12}, Colliot-Th\'{e}l\`{e}ne, Parimala, and Suresh showed that the function $(1-x)$ is not a square in $F'$, but is a square in $F'_v$ for every discrete valuation $v$ on $F'$. Therefore there is a two-dimensional counterexample to the local-global principle for isotropy of quadratic forms over $F'$ with respect to the set of all discrete valuations on $F'$. By \cite[Theorem~9.11(a)]{hhk15}, this implies that the reduction graph of a regular model of $F'$ is not a tree, hence $m(F') = 2$ by Lemma~\ref{loop implies m=2}. If the second condition in the definition of $m_s$ considered \textit{all} finitely generated field extensions of transcendence degree one, then the strong $m$-invariant of $\Q_p$ would be 1, and therefore Theorem \ref{strong m of cdvf} would be false since $m_s(\F_p) = 2$.
\end{example}

We begin by investigating basic properties of $m_s(k)$.
\begin{lemma}
\label{strong m bounded by strong u}
Let $k$ be any field of characteristic $\ne 2$. Then $m_s(k) \leq u_s(k)$.
\end{lemma}
\begin{proof}
Let $k'/k$ be any finite extension. Then
\begin{itemize}
	\item $m(k') \leq u(k') \leq u_s(k)$,
	
	\item $m(k'(t)) \leq u(k'(t)) \leq 2u_s(k)$.
\end{itemize}
Therefore, by definition, $m_s(k) \leq u_s(k)$.
\end{proof}
\begin{cor}
\label{if u_s = 1}
For a field $k$ of characteristic $\ne 2$, if $u_s(k) = 1$, then $m_s(k) = 1$.
\end{cor}
\begin{proof}
For any field $E$ of characteristic $\ne 2$, $m(E) \geq 1$, and by Lemma \ref{m at least 2 for fg trdeg 1}, $m(E(t)) \geq 2$, so $m_s(k) \geq 1$. By Lemma \ref{strong m bounded by strong u} we have $m_s(k) \leq u_s(k) = 1$, hence $m_s(k) = 1$ as claimed.
\end{proof}
The next lemma shows that the analogue of Proposition \ref{separated by power of 2} holds for $m_s$ and $u_s$.
\begin{lemma}
\label{strong m and strong u power of 2}
Let $k$ be any field of characteristic $\ne 2$. If $n \geq 0$ is the largest integer such that $2^n \leq u_s(k)$, then $m_s(k) \leq 2^n \leq u_s(k)$.
\end{lemma}
\begin{proof}
Since $n$ is the largest integer such that $2^n \leq u_s(k)$, we have $u_s(k) < 2^{n+1}$. Therefore for all finite extensions $k'/k$ we have $u(k') < 2^{n+1}$ and $u(k'(t)) < 2^{n+2}$. This, in turn, implies that $m(k') \leq 2^n$ and $m(k'(t)) \leq 2^{n+1}$ by Proposition \ref{separated by power of 2}. Thus, by definition, $m_s(k) \leq 2^n$.
\end{proof}
\begin{cor}
\label{if strong m = strong u}
Let $k$ be a field of characteristic $\ne 2$ such that $m_s(k) = u_s(k) < \infty$. Then $m_s(k) = u_s(k) = 2^n$ for some integer $n \geq 0$. Moreover, for all finite field extensions $k'/k$ we have
\begin{enumerate}[label=(\alph*)]
	\item $m(k') = u(k') = 2^n$,
	
	\item $m(k'(t)) = u(k'(t)) = 2^{n+1}$.
\end{enumerate}
\end{cor}
\begin{proof}
Let $n \geq 0$ be the largest integer such that $2^n \leq u_s(k)$. Then $m_s(k) \leq 2^n \leq u_s(k)$ by Lemma~\ref{strong m and strong u power of 2}. So, because $m_s(k) = u_s(k)$, we must have $m_s(k) = u_s(k) = 2^n$.

Now let $k'/k$ be any finite field extension. The remaining claims of the corollary follow from these inequalities:

(a) $2^n = m_s(k) \leq m(k') \leq u(k') \leq u_s(k) = 2^n$.
	
(b) $2^{n+1} = 2m_s(k) \leq m(k'(t)) \leq u(k'(t)) \leq 2u_s(k) = 2^{n+1}$.
\end{proof}

By studying how $u_s$ and $m_s$ change under finite extension, we see that if $m_s = u_s$, then these invariants are stable under finite extension.
\begin{lemma}
\label{strong invariants under finite extension}
Let $k$ be any field of characteristic $\ne 2$, and let $k'/k$ be any finite field extension. Then
\[
	m_s(k) \leq m_s(k') \leq u_s(k') \leq u_s(k).
\]
In particular, if $m_s(k) = u_s(k)$, then $m_s(k) = m_s(k') = u_s(k') = u_s(k)$.
\end{lemma}
\begin{proof}
The second claim follows immediately from the first, so it suffices to prove the first claim. By Lemma \ref{strong m bounded by strong u}, we know $m_s(k') \leq u_s(k')$. It therefore suffices to prove the first and third inequalities.

Let $k''$ be any finite field extension of $k'$. Then since $k''$ is also a finite extension of $k$, we have $m(k'') \geq m_s(k)$ and $m(k''(t)) \geq 2m_s(k)$, so $m_s(k') \geq m_s(k)$ by definition. Moreover, we have $u(k'') \leq u_s(k)$ by the definition of $u_s(k)$.

Now let $K'$ be any finitely generated field extension of $k'$ of transcendence degree one. Then~$K'$ is also a finitely generated field extension of $k$ of transcendence degree one, so $u(K') \leq 2u_s(k)$. Therefore $u_s(k') \leq u_s(k)$ by definition, completing the proof of the lemma.
\end{proof}

Let $T$ be a complete discrete valuation ring with residue field $k$ of characteristic $\ne 2$ and fraction field $K$. Using Springer's Theorem, one can show that $u(K) = 2u(k)$, and \cite[Theorem~4.10]{hhk09} states that the same conclusion holds for the strong $u$-invariant; i.e., $u_s(K) = 2u_s(k)$. We are now ready to prove
\begin{theorem}
\label{strong m of cdvf}
Let $T$ be a complete discrete valuation ring with fraction field $K$ and residue field~$k$ of characteristic $\ne 2$. If $m_s(k) = u_s(k)$, then $m_s(K) = 2m_s(k)$.
\end{theorem}
A main ingredient in the proof of Theorem \ref{strong m of cdvf} is the following result.
\begin{prop}
\label{lower bound on strong m}
Let $K$ be a complete discretely valued field with residue field $k$ of characteristic $\ne 2$, and suppose that $m_s(k) \geq n$ for some integer $n \geq 1$. Then $m_s(K) \geq 2n$.
\end{prop}
\begin{proof}
By definition, since $m_s(k) \geq n$, for all finite field extensions $k'/k$ we have
\[
	m(k') \geq n \text{ and } m(k'(t)) \geq 2n.
\]

Now let $K'/K$ be any finite extension of $K$. The field $K'$ is a finite extension of the complete discretely valued field $K$, so $K'$ is also a complete discretely valued field, with residue field~$k'$ a finite field extension of $k$. By Corollary \ref{m-inv of cdvf} we have $m(K') = 2m(k') \geq 2n$. So to complete the proof, we must show that $m(K'(x)) \geq 4n$.

Let $S$ be the valuation ring of $K'$, let $F = K'(x)$, and let $q$ be any anisotropic universal quadratic form over $F$. We want to show that $\dim q \geq 4n$. By \cite[Proposition~3.4]{becher-leep} we know $m(F) > 2$, thus $\dim q \geq 3$. By taking a suitable blow-up $\mathscr{X} \to \P^1_S$ we obtain a normal crossings $q$-model $\mathscr{X}$ of~$F$ with closed fiber $X$. By our choice of model $\mathscr{X}$, for all points $P \in X$, since~$q$ is universal over $F$, it is also universal over $F_P$ by Lemma \ref{universal over F implies universal over F_P}. Moreover, because $q$ is anisotropic over~$F$ with $\dim q \geq 3$, by \cite[Theorem~9.3]{hhk15} there must be a particular point $P_* \in X$ such that~$q$ is anisotropic over $F_{P_*}$. Thus $q$ is anisotropic and universal over $F_{P_*}$. This point $P_*$ is either a closed point or the generic point of an irreducible component of $X$.

If $P_*$ is a closed point, then because $q$ is anisotropic and universal over $F_{P_*}$, by Corollary~\ref{m of F_P for P closed} we have $\dim q \geq 4m(\kappa(P_*))$. The field $\kappa(P_*)$ is a finite extension of $k$, thus $m(\kappa(P_*)) \geq n$ by our assumption on $k$. This implies that $\dim q \geq 4n$ in this case.

Now suppose that $P_*$ is the generic point $\eta$ of an irreducible component $X_{\eta}$ of the closed fiber of~$\mathscr{X}$. In this case the field $F_{P_*} = F_{\eta}$ is a complete discretely valued field whose residue field~$\kappa(\eta)$ is the residue field of the local ring $\mathcal{O}_{\mathscr{X}, \eta}$ of $\mathscr{X}$ at $\eta$. Since $\mathscr{X}$ is a blow-up of the regular $S$-curve~$\P^1_S$, the irreducible component $X_{\eta}$ is either an exceptional divisor or birational to the closed fiber of~$\P^1_S$, hence $\kappa(\eta) \cong k''(t)$ for some finite extension $k''$ of $k'$. The field $k'$ is a finite extension of~$k$, so $k''$ is also a finite extension of $k$. Since $m_s(k) \geq n$ by assumption, by applying Corollary~\ref{m-inv of cdvf} we have $m(F_{\eta}) = 2m(\kappa(\eta)) = 2m(k''(t)) \geq 4n$. Since $q$ is anisotropic and universal over $F_{\eta}$, we conclude $\dim q \geq m(F_{\eta}) \geq 4n$.

We have shown that any anisotropic universal quadratic form over $K'(x)$ must have dimension at least $4n$, so we conclude that $m(K'(x)) \geq 4n$, as desired.
\end{proof}

\begin{proof}[Proof of Theorem \ref{strong m of cdvf}]
By Proposition \ref{lower bound on strong m}, we have $m_s(K) \geq 2m_s(k)$. By Lemma \ref{strong m bounded by strong u}, we have $m_s(K) \leq u_s(K)$. By \cite[Theorem~4.10]{hhk09}, we have $u_s(K) = 2u_s(k)$. Finally, since we assumed $m_s(k) = u_s(k)$, it follows that $m_s(K) \leq 2m_s(k)$, which completes the proof.
\end{proof}
Theorem \ref{strong m of cdvf} allows us to calculate exact values of the $m$-invariant of rational function fields in one variable over certain complete discretely valued fields, as illustrated by the following corollary.
\begin{cor}\label{some m examples}
\begin{enumerate}[label=(\alph*)]
	\item If $p$ is an odd prime, then $m(\Q_p(x)) = 8$.
	
	\item If $k$ is an algebraically closed field of characteristic $\ne 2$, then $m(k(x)((y))(z)) = 8$.
	
	\item If $p$ is an odd prime and $r \geq 1$ is any positive integer, then $m(\F_p((t_1)) \cdots ((t_r))(x)) = 2^{r+2}$.
\end{enumerate}
\end{cor}
\begin{proof}
	(a) For any prime $p \ne 2$, $m_s(\F_p) = u_s(\F_p) = 2$, so $m_s(\Q_p) = 2m_s(\F_p) = 4 = u_s(\Q_p)$ by Theorem \ref{strong m of cdvf}. Hence $m(\Q_p(x)) = 8$ by Corollary \ref{if strong m = strong u}(b).
	
	(b) By Examples \ref{strong m examples}(3), $m_s(k(x)) = u_s(k(x)) = 2$. Thus $m_s(k(x)((y))) = u_s(k(x)((y))) = 4$ by Theorem \ref{strong m of cdvf}. Finally, we conclude that $m(k(x)((y))(z)) = 8$ by Corollary \ref{if strong m = strong u}(b).
	
	(c) For any prime $p \ne 2$, we have $m_s(\F_p) = u_s(\F_p) = 2$. So for any integer $r \geq 1$, applying Theorem~\ref{strong m of cdvf} inductively, we have $m_s(\F_p((t_1)) \cdots ((t_r))) = u_s(\F_p((t_1)) \cdots ((t_r))) = 2^{r+1}$. Therefore $m(\F_p((t_1)) \cdots ((t_r))(x)) = 2^{r+2}$ by Corollary~\ref{if strong m = strong u}(b).
\end{proof}
We conclude this section by calculating the $u$-invariant of a one-variable function field $F$ over a field $k$ with $m_s(k) = u_s(k) < \infty$.
\begin{prop}
\label{u-inv if m_s = u_s}
Let $k$ be a field of characteristic $\ne 2$ such that $m_s(k) = u_s(k) < \infty$. Let~$F$ be a one-variable function field over $k$. Then $u(F) = 2u_s(k)$.
\end{prop}
\begin{proof}
The proof of this result closely mirrors that of \cite[Corollary~4.13(c)]{hhk09}.

By definition we have $u(F) \leq 2u_s(k)$. To prove the reverse inequality, let $\mathcal{X}$ be a normal (equivalently, regular) $k$-curve with function field~$F$, and choose a closed point $\xi$ on $\mathcal{X}$. The local ring $\mathcal{O}_{\mathcal{X}, \xi}$ of $\mathcal{X}$ at $\xi$ is a discrete valuation ring with fraction field $F$ and residue field $\kappa(\xi)$, which is a finite extension of $k$. By Corollary \ref{if strong m = strong u}(a) we have $u(\kappa(\xi)) = u_s(k)$. So, applying \cite[Lemma~4.9]{hhk09} to $\mathcal{O}_{\mathcal{X}, \xi}, F$, we have $u(F) \geq 2u(\kappa(\xi)) = 2u_s(k)$, giving us the desired inequality.
\end{proof}

\section{All anisotropic universal quadratic forms over a semi-global field}
\label{all universal forms}
The goal of this section is to study the set $\au(F)$ of dimensions of anisotropic universal quadratic forms over a given semi-global field $F$ (see the definition preceding Lemma \ref{anisotropic universal residue forms}). The main result of this section is Theorem \ref{au equals union}, which states that we can use information coming from a regular model of $F$ to determine $\au(F)$. We now introduce notation that will be used throughout Sections~\ref{all universal forms} and~\ref{classifying curves}.

\begin{notation}
Let $F$ be a semi-global field, and let $\mathscr{X}$ be a regular model of $F$ with closed fiber~$X$ and reduced closed fiber $X^{\text{red}}$. If $\eta$ is the generic point of an irreducible component of $X$ or $X^{\text{red}}$, then we let $\kappa(\eta)$ denote the residue field of the local ring $\mathcal{O}_{\mathscr{X}, \eta}$ of $\mathscr{X}$ at $\eta$. For a scheme $\mathscr{Y}$, a point $y \in \mathscr{Y}$, and an element $f \in \mathcal{O}_{\mathscr{Y}, y}$, we let $\overline{f}_y \in \kappa(y)$ denote the image of $f$ in the residue field $\kappa(y)$ of the local ring~$\mathcal{O}_{\mathscr{Y},y}$. Furthermore, we let $\widehat{\mathcal{O}}_{\mathscr{Y}, y}$ denote the completion of the local ring $\mathcal{O}_{\mathscr{Y}, y}$ with respect to its maximal ideal.
\end{notation}

To begin, we prove a lemma that will be used several times throughout this section.
\begin{lemma}
\label{u-inv of residue fields}
Let $T$ be a complete discrete valuation ring with residue field $k$ of characteristic~$\ne 2$ such that $m_s(k) = u_s(k) = 2^n$ for some integer $n \geq 0$. If $\eta$ is the generic point of an irreducible component of the (reduced) closed fiber of a regular connected projective $T$-curve, then $u(\kappa(\eta)) =$~$2^{n+1}$.
\end{lemma}
\begin{proof}
The field $\kappa(\eta)$ is a finitely generated field extension of transcendence degree one over $k$. Since $m_s(k) = u_s(k) = 2^n$, we have $u(\kappa(\eta)) = 2^{n+1}$ by Proposition \ref{u-inv if m_s = u_s}.
\end{proof}
The first step in proving Theorem \ref{au equals union} is proving
\begin{prop}
\label{au contained in union}
Let $T$ be a complete discrete valuation ring with fraction field $K$ and residue field $k$ of characteristic $\ne 2$ such that $m_s(k) = u_s(k) < \infty$. Let $\mathscr{X}$ be a regular connected projective $T$-curve with closed fiber $X$. Let $X_1, \ldots, X_s$ be the irreducible components of $X$, and for $1 \leq i \leq s$ let $\eta_i$ be the generic point of $X_i$. Let $\Gamma$ be the reduction graph of $\mathscr{X}$, and let $F$ be the function field of $\mathscr{X}$. Then
\[
	\emph{AU}(F) \subseteq \begin{cases}
	\{2\} \cup \bigcup_{i = 1}^s \left\{r_1 + r_2 \mid r_1, r_2 \in \emph{AU}(\kappa(\eta_i))\right\} &\text{ if $\Gamma$ is not a tree}, \\
	\bigcup_{i = 1}^s \left\{r_1 + r_2 \mid r_1, r_2 \in \emph{AU}(\kappa(\eta_i)) \right\} &\text{ if $\Gamma$ is a tree}.
	\end{cases}
\]
\end{prop}
\begin{proof}
If $\Gamma$ is not a tree, then by Lemma \ref{loop implies m=2}, $m(F) = 2 \in \au(F)$. If $\Gamma$ is a tree, then $m(F) > 2$, so $2 \not\in \au(F)$. To complete the proof, we must show that if $q$ is any anisotropic universal quadratic form over~$F$ with $\dim q \geq 3$, then $\dim q \in \left\{r_1 + r_2 \mid r_1, r_2 \in \au(\kappa(\eta_i))\right\}$ for some $i \in \{1, \ldots, s\}$.

We first note that since $m_s(k) = u_s(k) < \infty$, Corollary~\ref{if strong m = strong u} implies that $m_s(k) = u_s(k) = 2^n$ for some integer $n \geq 0$. Thus $u(\kappa(\eta_i)) = 2^{n+1}$ for all $i = 1, \ldots, s$ by Lemma \ref{u-inv of residue fields}.

Let $q$ be any anisotropic universal quadratic form over $F$ with $\dim q \geq 3$. Because $\Char F \ne 2$, we may assume that $q$ is a diagonal form $\langle a_1, \ldots, a_d \rangle$ with each $a_i \in F^{\times}$. By taking a suitable blow-up $\mathscr{X}' \to \mathscr{X}$ we obtain a normal crossings $q$-model $\mathscr{X}'$ of $F$
with closed fiber $X'$ (see Definition \ref{normal crossings q-model} and Remark \ref{normal crossings q-models exist}).

Because $q$ is universal over $F$ and because of our choice of model $\mathscr{X}'$, it follows that $q$ is universal over~$F_P$ for all points $P \in X'$ by Lemma \ref{universal over F implies universal over F_P}. Moreover, since $\dim q \geq 3$ and $q$ is anisotropic over~$F$, by \cite[Theorem~9.3]{hhk15} there must be a point $P_{*} \in X'$ such that $q$ is anisotropic (and universal) over $F_{P_{*}}$.

First suppose that $P_*$ is a closed point. Then because $q$ is anisotropic and universal over~$F_{P_*}$, by Corollary \ref{m of F_P for P closed} we have $\dim q \geq 4m(\kappa(P_*))$. The field $\kappa(P_*)$ is a finite extension of~$k$, thus $m(\kappa(P_*)) = 2^n$ by Corollary \ref{if strong m = strong u}(a). Hence $\dim q \geq 2^{n+2}$. Since $m_s(k) = u_s(k) = 2^n$, we have $m_s(K) = u_s(K) = 2^{n+1}$ by Theorem \ref{strong m of cdvf}. The field $F$ is a one-variable function field over $K$, so $u(F) = 2^{n+2}$ by Proposition \ref{u-inv if m_s = u_s}, and since $q$ is anisotropic over $F$ with $\dim q \geq u(F) = 2^{n+2}$, we conclude $\dim q = 2^{n+2}$. Now, for all $i = 1, \ldots, s$, as we observed above, $2^{n+1} \in \au(\kappa(\eta_i))$. Therefore
\[
	\dim q = 2^{n+1} + 2^{n+1} \in \left\{r_1 + r_2 \mid r_1, r_2 \in \au(\kappa(\eta_i))\right\}
\]
for all $i = 1, \ldots, s$, proving the claim if $P_* \in X'$ is a closed point.

Now suppose that the point $P_* \in X'$ is the generic point $\xi$ of an irreducible component $X'_{\xi}$ of~$X'$. The form $q$ is anisotropic and universal over $F_{P_*} = F_{\xi}$, so $\dim q \in \au(F_{\xi})$. In this case, $F_{\xi}$ is a complete discretely valued field with residue field $\kappa(\xi)$, so by Corollary \ref{au of cdvf},
\[
	\au(F_{\xi}) = \left\{r_1 + r_2 \mid r_1, r_2 \in \au(\kappa(\xi))\right\}.
\]
Since $\mathscr{X}'$ is a blow-up of the regular $T$-curve $\mathscr{X}$, the irreducible component $X'_{\xi}$ is either an exceptional divisor, in which case $X'_{\xi} \simeq \mathbb{P}_{k'}^1$ for some finite extension $k'$ of $k$ by, e.g., \cite[Theorem~9.3.8]{liu}, or $X'_{\xi}$ is birational to an irreducible component $X_i$ of $X$.

If $X_{\xi}' \simeq \mathbb{P}^1_{k'}$, then $\kappa(\xi) \cong k'(t)$. The field $k'$ is a finite extension of $k$, and since $m_s(k) = u_s(k) =$~$2^n$, we have $m(k'(t)) = u(k'(t)) = 2^{n+1}$ by Corollary \ref{if strong m = strong u}(b). This implies that $m(F_{\xi}) = u(F_{\xi}) = 2^{n+2}$, and as we saw above, $2^{n+2} = 2u(\kappa(\eta_i))$ for all $i = 1, \ldots, s$. So in this case,
\[
	\au(F_{\xi}) = \left\{2^{n+2}\right\} \subseteq \left\{r_1 + r_2 \mid r_1, r_2 \in \au\left(\kappa(\eta_i)\right)\right\}
\]
for all $i = 1, \ldots, s$.

If $X'_{\xi}$ is birational to an irreducible component $X_i$ of $X$, then $\kappa(\xi) \cong \kappa(\eta_i)$. This implies that $\au(\kappa(\xi)) = \au(\kappa(\eta_i))$. Therefore $\au(F_{\xi}) = \left\{r_1 + r_2 \mid r_1, r_2 \in \au(\kappa(\eta_i))\right\}$ by Corollary \ref{au of cdvf}.

In either case of $X_{\xi}'$, we have shown that $\dim q \in \au(F_{\xi}) \subseteq \left\{r_1 + r_2 \mid r_1, r_2 \in \au(\kappa(\eta_i))\right\}$ for some $i \in \{1, \ldots, s\}$, so the proof is complete.
\end{proof}

The next step in proving Theorem \ref{au equals union} is showing that the containment opposite to the one in Proposition \ref{au contained in union} holds. This step is constructive, and one of the key ingredients is Proposition \ref{lifting universal forms on nice closed fiber}. Proposition \ref{lifting universal forms on nice closed fiber} proves the existence of specified lifts of anisotropic universal quadratic forms from the reduced closed fiber of a regular curve over a complete discrete valuation ring to the curve itself, and relies on several preliminary results.

\begin{lemma}
\label{universal form on component of closed fiber}
Let $k$ be a field of characteristic $\ne 2$ such that $m_s(k) = u_s(k) = 2^n$ for some integer $n \geq 1$. Let $X$ be a regular irreducible $k$-curve with function field $K$, and let $q$ be a universal quadratic form over $K$ with $2 \leq \dim q < 2^{n+1}$. Then for all closed points $P \in X$, $q$ is isotropic over the completion $K_{v_P}$ of $K$ with respect to the discrete valuation $v_P$ induced on $K$ by~$P$.
\end{lemma}
\begin{proof}
Since $q$ is universal over $K$, for all closed points $P \in X$, $q$ must also be universal over~$K_{v_P}$ by Lemma \ref{globally universal implies locally universal}. For each closed point $P \in X$, the field $K_{v_P}$ is a complete discretely valued field whose residue field $\kappa(P)$ is a finite extension of $k$. Since $m_s(k) = u_s(k) = 2^n$, Corollary~\ref{if strong m = strong u}(a) implies that $m(\kappa(P)) = 2^n$. Therefore $m(K_{v_P}) = 2m(\kappa(P)) = 2^{n+1}$ by Corollary \ref{m-inv of cdvf}. Since $q$ is universal over $K_{v_P}$ with $\dim q < 2^{n+1} = m(K_{v_P})$, it follows that $q$ must be isotropic over $K_{v_P}$.
\end{proof}

The next result is analogous to the Chinese Remainder Theorem. For a field $K$ equipped with a non-trivial discrete valuation $v$, we let $\mathcal{O}_v$ denote the valuation ring of $v$ and let $\kappa_v$ denote the residue field of $\mathcal{O}_v$. 
\begin{lemma}
\label{lifts with prescribed residue}
Let $K$ be a field, and let $v_1, \ldots, v_n$ be non-trivial pairwise non-equivalent discrete valuations on $K$. For some $1 \leq j \leq n$, let $a_1 \in \kappa_{v_1}^{\times}, a_2 \in \kappa_{v_2}^{\times}, \ldots, a_j \in \kappa_{v_j}^{\times}$ be given. Then there exist elements $\widetilde{a}, \widetilde{b} \in K^{\times}$ such that $\widetilde{a} \in \mathcal{O}_{v_i}^{\times}$ for all $i = 1, \ldots, n$, $\widetilde{b} \in \mathcal{O}_{v_i}^{\times}$ for all $i = 1, \ldots, j$ and $v_i(\widetilde{b}) = 1$ for all $j < i \leq n$, and for each $i = 1, \ldots, j$ the images of $\widetilde{a}$ and $\widetilde{b}$ in $\kappa_{v_i}^{\times}$ are $a_i$.
\end{lemma}
\begin{proof}
For each $i = 1, \ldots, j$ let $\widetilde{a}_i \in \mathcal{O}_{v_i}^{\times}$ be a unit lift of $a_i$ to $K$. Then by \cite[Chapter VI, Theorem~18]{zar-sam}, there are elements $\widetilde{a}, \widetilde{b} \in K^{\times}$ such that
\[
	v_i\left(\widetilde{a} - \widetilde{a}_i\right) = v_i(\widetilde{b} - \widetilde{a}_i) = 1 \text{ for all $i = 1, \ldots, j$, and } v_i\left(\widetilde{a}\right) = 0, v_i(\widetilde{b}) = 1 \text{ for all $j < i \leq n$}.
\]
The element $\widetilde{a}$ satisfies $v_i(\widetilde{a}) = 0$ for all $i = 1, \ldots, n$, thus $\widetilde{a} \in \mathcal{O}_{v_i}^{\times}$ for all $i = 1, \ldots, n$. Furthermore, the element $\widetilde{b}$ satisfies $v_i(\widetilde{b}) = 0$ (hence $\widetilde{b} \in \mathcal{O}_{v_i}^{\times}$) for all $i = 1, \ldots, j$, and $v_i(\widetilde{b}) = 1$ for all $j < i \leq n$. Next, for each $i = 1, \ldots, j$, since $\widetilde{a}, \widetilde{b} \in \mathcal{O}_{v_i}^{\times}$ and $v_i(\widetilde{a} - \widetilde{a}_i) = v_i(\widetilde{b} - \widetilde{a}_i) = 1$, the images of $\widetilde{a}$, $\widetilde{b}$, and~$\widetilde{a}_i$ in $\kappa_{v_i}^{\times}$ must all agree. So for each $i = 1, \ldots, j$ the images of $\widetilde{a}$ and $\widetilde{b}$ in $\kappa_{v_i}^{\times}$ are $a_i$.
\end{proof}

The next several results (Lemma \ref{particular isotropic vector} -- Corollary \ref{surjective map on global sections}) are used to construct certain lifts of universal quadratic forms on the reduced closed fiber of a regular model of a semi-global field.
\begin{lemma}
\label{particular isotropic vector}
Let $q = \langle a_1, \ldots, a_n \rangle$ be a regular quadratic form over a field $k$ of characteristic $\ne 2$. If $q$ is isotropic, then there is an isotropic vector $x = (x_1, \ldots, x_n)$ for $q$ with $x_1 = 1$.
\end{lemma}
\begin{proof}
Let $x' = (x_1', \ldots, x_n')$ be any isotropic vector for $q$. If $x_1' \ne 0$, then let $x = (1, \frac{x_2'}{x_1'}, \ldots, \frac{x_n'}{x_1'} )$. If $x_1' = 0$, then $q' = \langle a_2, \ldots, a_n \rangle$ is regular and isotropic, hence universal. So there exist $x_2, \ldots, x_n \in k$ such that $q'(x_2, \ldots, x_n) = -a_1$, and $x = (1, x_2, \ldots, x_n)$ is our desired isotropic vector.
\end{proof}

\begin{lemma}
\label{represents 1 over residue field}
Let $k$ be a field of characteristic $\ne 2$, and let $X$ be a regular irreducible $k$-curve with function field $K$. Let $\mathcal{P}$ be a non-empty finite set of closed points of $X$, and for each $P \in \mathcal{P}$ let $\pi_P \in K^{\times}$ be a uniformizer for the discrete valuation $v_P$ induced on $K$ by $P$. Suppose we are given elements $a_2, \ldots, a_d \in K^{\times}$ such that
\begin{itemize}
	\item for each $P \in \mathcal{P}$ and each $i = 2, \ldots, d$ there is some $a_{i,P}' \in \mathcal{O}_{X,P}^{\times}$ such that either $a_i = a_{i,P}'$ or $a_i = \pi_P a_{i,P}'$ (i.e., $v_P(a_i) \in \{0, 1\}$), and
	
	\item the quadratic form over $K$ defined by $q =  \langle -1, a_2, \ldots, a_d \rangle$ is isotropic over $K_{v_P}$ for all $P \in \mathcal{P}$.
\end{itemize}
Then for each $P \in \mathcal{P}$ there are elements $\overline{x}_{2, P}, \ldots, \overline{x}_{d, P} \in \kappa(P)$ such that $\sum_{i = 2}^d \overline{a}_{i,P}' \overline{x}_{i,P}^2 = 1$.
\end{lemma}
\begin{proof}
Let $P \in \mathcal{P}$ be given, and for $j = 0, 1$, let $I_{j,P} = \left\{i \in \{2, \ldots, d\} \mid v_P(a_i) = j\right\}$. By our assumptions on the entries of $q$, we have $q = q_{1,P} \perp \pi_P \cdot q_{2,P}$, where the entries of $q_{1,P}$ are $-1$ and~$a_{i,P}'$ for $i \in I_{0,P}$, and the entries of $q_{2,P}$ are those $a_{i,P}'$ such that $i \in I_{1,P}$. In particular, all of the entries of $q_{1,P}$ and $q_{2,P}$ are units in $\mathcal{O}_{X,P}$.

By assumption, the form $q$ is isotropic over $K_{v_P}$. So by Springer's Theorem, one of the residue forms of $q$ must be isotropic over $\kappa(P)$. If $\overline{q}_{1,P} = \langle -1 \rangle \perp \overline{q}_{1,P}'$ is isotropic over $\kappa(P)$, then by Lemma~\ref{particular isotropic vector} we can find elements $\overline{x}_{i,P} \in \kappa(P)$ for $i \in I_{0,P}$ such that $\sum_{i \in I_{0,P}} \overline{a}_{i,P}' \overline{x}_{i,P}^2 = 1$. If we let $\overline{x}_{i,P} = 0$ for all $i \in I_{1,P}$, then we have $\sum_{i = 2}^d \overline{a}_{i,P}' \overline{x}_{i,P}^2 = 1$, as desired.

If $\overline{q}_{1,P}$ is anisotropic, then $\overline{q}_{2,P}$ is regular and isotropic, hence universal over $\kappa(P)$. Therefore~$\overline{q}_{2,P}$ represents 1, so there exist elements $\overline{x}_{i,P} \in \kappa(P)$ for $i \in I_{1,P}$ such that $\sum_{i \in I_{1,P}} \overline{a}_{i,P}' \overline{x}_{i,P}^2 =$~$1$. Letting $\overline{x}_{i,P} = 0$ for all $i \in I_{0,P}$, we again have $\sum_{i = 2}^d \overline{a}_{i,P}' \overline{x}_{i,P}^2 = 1$. 
\end{proof}

\begin{prop}
\label{defining functions on components}
Let $k$ be a field of characteristic $\ne 2$. Let $X$ be a connected curve over $k$ that is the union of $s \geq 2$ regular irreducible $k$-curves $X_1, \ldots, X_s$, with normal crossings. For $1 \leq i \leq s$ let~$\eta_i$ be the generic point of $X_i$, let $\kappa(\eta_i)$ be the function field of $X_i$, and let $\mathcal{P}_i$ be the non-empty finite set of closed points $P \in X_i$ such that $P \in X_i \cap X_j$ for some $j \ne i$. Suppose we are given $a_{1,2}, \ldots, a_{1,d} \in$~$\kappa(\eta_1)^{\times}$ such that $v_P(a_{1,r}) \in \{0,1\}$ for all $r = 2, \ldots, d$ and the quadratic form $q_1 = $~$\langle -1, a_{1,2}, \ldots, a_{1,d} \rangle$ is isotropic over $\kappa(\eta_1)_{v_P}$ for all $P \in \mathcal{P}_1$. Then for each $i = 2, \ldots, s$ there are elements $a_{i,2}, \ldots, a_{i,d} \in \kappa(\eta_i)^{\times}$ such that
\begin{enumerate}[label=(\alph*)]
	\item the quadratic form $q_i = \langle -1, a_{i,2}, \ldots, a_{i,d} \rangle$ is isotropic over $\kappa(\eta_i)$,
	
	\item if $P \in \mathcal{P}_i \cap \mathcal{P}_j$, then $\overline{a}_{i,r,P} = \overline{a}_{j,r,P} \in \kappa(P)$ for all $r = 2, \ldots, d$,
	
	\item if $P \in \mathcal{P}_i \cap \mathcal{P}_j$ for $i,j \ne 1$, then $a_{i,r} \in \mathcal{O}_{X_i, P}^{\times}$ and $a_{j,r} \in \mathcal{O}_{X_j,P}^{\times}$ for all $r = 2, \ldots, d$.
\end{enumerate}
\end{prop}
\begin{proof}
The idea of the proof is as follows. We will work with one $X_i$ at a time, first finding elements of $\kappa(\eta_i)^{\times}$ that satisfy conditions (a) and (c), and then we scale these elements by certain squares so that all three conditions are satisfied. We then iterate this procedure as many times as necessary.

For each $P \in \mathcal{P}_1$ let $\pi_{1,P} \in \kappa(\eta_1)^{\times}$ be a uniformizer for the discrete valuation $v_P$ on~$\kappa(\eta_1)$. Then for each $r = 2, \ldots, d$ and $P \in \mathcal{P}_1$, since $v_P(a_{1,r}) \in \{0, 1\}$ by assumption, there is some element $a_{1,r,P}' \in \mathcal{O}_{X_1, P}^{\times}$ such that either $a_{1,r} = a_{1,r,P}'$ or $a_{1,r} = \pi_{1,P}a_{1,r,P}'$. Now, for each $P \in$~$\mathcal{P}_1$, since we assumed that $q_1$ is isotropic over $\kappa(\eta_1)_{v_P}$, there are elements $\overline{x}_{1,2,P}, \ldots, \overline{x}_{1,d,P} \in$~$\kappa(P)$ such that $\sum_{r = 2}^d \overline{a}_{1,r,P}' \overline{x}_{1,r,P}^2 = 1 \in \kappa(P)$ by Lemma \ref{represents 1 over residue field}.

Next, choose a curve $X_i$, $i \geq 2$, such that $X_1 \cap X_i \ne \varnothing$. This choice is possible because~$X$ is connected. By our normal crossings assumption, we can write $\mathcal{P}_i$ as a disjoint union $\mathcal{P}_i = \mathcal{P}_i' \cup \mathcal{P}_i''$, where $\mathcal{P}_i' = \mathcal{P}_i \cap \mathcal{P}_1 = \{P \in X_i \mid P \in X_i \cap X_1\}$ is non-empty. By the previous paragraph, for each $P \in \mathcal{P}_i'$ and each $r = 2, \ldots, d$ there is an $\overline{x}_{1,r,P} \in \kappa(P)$ such that $\sum_{r=2}^d \overline{a}_{1,r,P}' \overline{x}_{1,r,P}^2 = 1$. By Lemma \ref{lifts with prescribed residue} there are elements $a_{i,2}', \ldots, a_{i,d}' \in$~$\kappa(\eta_i)^{\times}$ such that, for all $r = 2, \ldots, d$, $a_{i,r}' \in \mathcal{O}_{X_i, P}^{\times}$ for all $P \in \mathcal{P}_i$ and $\overline{a}_{i,r,P}' = \overline{a}_{1,r,P}' \in$~$\kappa(P)^{\times}$ for all $P \in$~$\mathcal{P}_i'$. Furthermore, there are elements $x_{i,2}, \ldots, x_{i,d} \in \mathcal{O}_{X_i,P} \subset \kappa(\eta_i)$ such that $\overline{x}_{i,r,P} = \overline{x}_{1,r,P} \in$~$\kappa(P)$ for all $r = 2, \ldots, d$ and all $P \in \mathcal{P}_i'$, $x_{i,2} \in \mathcal{O}_{X_i,P}^{\times}$ with $\overline{x}_{i,2,P} = 1 \in \kappa(P)$ for all $P \in \mathcal{P}_i''$, and $v_P(x_{i,r}) = 1$ for all $P \in \mathcal{P}_i''$ and all $2 < r \leq d$. Now let $z_i = \sum_{r = 2}^d a_{i,r}' x_{i,r}^2 \in \kappa(\eta_i)^{\times}$. By our choice of $a_{i,r}'$ and $x_{i,r}$, if $P \in \mathcal{P}_i'$, then $\overline{z}_{i,P} = 1$, and if $P \in \mathcal{P}_i''$, then $\overline{z}_{i,P} = \overline{a}_{i,2,P}' \in \kappa(P)^{\times}$. In particular, $z_i \in \mathcal{O}_{X_i, P}^{\times}$ for all $P \in \mathcal{P}_i$.

Consider the quadratic form $q_i' = \langle -1, z_i^{-1}a_{i,2}', \ldots, z_i^{-1} a_{i,d}' \rangle$ over $\kappa(\eta_i)$. The form $q_i'$ is isotropic over~$\kappa(\eta_i)$ with isotropic vector $(1, x_{i,2}, \ldots, x_{i,d})$, and $z_i^{-1} a_{i,r}' \in \mathcal{O}_{X_i,P}^{\times}$ for all $P \in \mathcal{P}_i$ and all $r =$~$2, \ldots, d$. Next, by Lemma \ref{lifts with prescribed residue}, for each $P \in \mathcal{P}_i'$ we can find an element $\pi_{i,P} \in \kappa(\eta_i)^{\times}$ such that $v_P(\pi_{i,P}) = 1 \text{ and } v_{P'}(\pi_{i,P}) = 0 \text{ with } \overline{\pi}_{i,P,P'} = 1 \in \kappa(P')$ for all $P' \in \mathcal{P}_i \setminus \{P\}$. For each $r = 2, \ldots, d$ let $S_{i,r} = \left\{P \in \mathcal{P}_i' \mid v_P(a_{1,r}) = 1 \right\}$, and let $\pi_{i,r} = \prod_{P \in S_{i,r}} \pi_{i,P}$, where this product is taken to be 1 if $S_{i,r}$ is empty. Then the quadratic form $q_i = \langle -1, z_i^{-1}a_{i,2}' \pi_{i,2}^2, \ldots, z_i^{-1}a_{i,d}' \pi_{i,d}^2 \rangle$ over $\kappa(\eta_i)$ is isometric to $q_i'$ over $\kappa(\eta_i)$, hence $q_i$ is isotropic over $\kappa(\eta_i)$. 

Finally, for each $r = 2, \ldots, d$ let $a_{i,r} = z_i^{-1}a_{i,r}' \pi_{i,r}^2 \in \kappa(\eta_i)^{\times}$. Then for each $r = 2, \ldots, d$ and each $P \in \mathcal{P}_i$ we have $\overline{a}_{i,r,P} = \overline{a}_{1,r,P} \in \kappa(P)$ if $P \in \mathcal{P}_i'$, and $\overline{a}_{i,r,P} = \overline{a}_{i,2,P}'^{-1} \overline{a}_{i,r,P}' \in \kappa(P)^{\times}$ if $P \in \mathcal{P}_i''$. The elements $a_{i,2}, \ldots, a_{i,d}$ are therefore the desired elements of $\kappa(\eta_i)^{\times}$, and if $s = 2$ then we are done. We note here that if $\mathcal{P}_i'' \ne \varnothing$ (which can occur only if $s > 2$), then for all $r = 2, \ldots, d$ and all $P \in \mathcal{P}_i''$ we have $v_P(a_{i,r}) = 0$. Furthermore, since $q_i$ is isotropic over $\kappa(\eta_i)$, it follows that $q_i$ is isotropic over $\kappa(\eta_i)_{v_P}$ for all $P \in \mathcal{P}_i''$. So by Lemma \ref{represents 1 over residue field}, for each $P \in \mathcal{P}_i''$ we can find elements $\overline{x}'_{i,2,P}, \ldots, \overline{x}'_{i,d,P} \in \kappa(P)$ such that $\sum_{r = 2}^d \overline{a}_{i,r,P} \overline{x}_{i,r,P}'^{2} = 1$.

Now suppose $s > 2$ and let $j \in \{2, \ldots, s\}$ be such that we have not yet found $a_{j,2}, \ldots, a_{j,d} \in$~$\kappa(\eta_j)^{\times}$ and the curve $X_j$ intersects at least one of $X_i$ and $X_1$. Again, such a choice is possible because $X$ is connected. We can then write $\mathcal{P}_j$ as a disjoint union $\mathcal{P}_j = \mathcal{P}_{j}' \cup \mathcal{P}_{j}''$, where $\mathcal{P}_{j}'$ consists of the closed points $P \in X_{j}$ such that $P \in X_{j} \cap X_i$ or $P \in X_j \cap X_1$. Then by repeating the same argument as that above with $X_j$ replacing $X_i$, $\mathcal{P}_j'$ replacing $\mathcal{P}_i'$, and $\mathcal{P}_j''$ replacing $\mathcal{P}_i''$, we can find the desired elements $a_{j,2}, \ldots, a_{j,d} \in \kappa(\eta_{j})^{\times}$. Repeating this process as many times as necessary, at each stage selecting a curve $X_{\ell}$ in the same fashion as we selected $X_j$ above, completes the proof.
\end{proof}

\begin{remark}
\label{form with small valuation}
Let $v_1, \ldots, v_n$ be non-trivial pairwise non-equivalent discrete valuations on a field $k$ of characteristic $\ne 2$. For each $i = 1, \ldots, n$ we can find an element $\pi_i \in k^{\times}$ such that $v_i(\pi_i) = 1$ and $v_j(\pi_i) = 0$ for all $j \ne i$. Let $a \in k^{\times}$ be arbitrary. For each $i = 1, \ldots, n$ let $r_i = -v_i(a)$ if $v_i(a)$ is even or $r_i = -v_i(a) + 1$ if $v_i(a)$ is odd. Then for all $i = 1, \ldots, n$, $r_i$ is even and $v_i(a \prod_{j = 1}^n \pi_j^{r_j}) \in \{0, 1\}$. So, given any elements $a_1, \ldots, a_d \in k^{\times}$, the quadratic form $q = \langle a_1, \ldots, a_d \rangle$ over $k$ is isometric to a quadratic form $q' = \langle a_1', \ldots, a_d' \rangle$ over $k$ whose entries satisfy $v_i(a_j') \in \{0, 1\}$ for all $i = 1, \ldots, n$ and all $j = 1, \ldots, d$. Since $q$ and $q'$ are isometric forms, we see that $q$ is, e.g., isotropic or universal if and only if $q'$ is isotropic or universal, respectively.
\end{remark}

\begin{lemma}
\label{surjective from closed fiber}
Let $X$ be a projective scheme over a field, and let $X^{\text{\emph{red}}}$ be the reduced scheme associated to $X$, with $\iota: X^{\text{\emph{red}}} \to X$ the corresponding closed immersion. Let $D$ be an ample divisor on $X$. Then there is an integer $n_0 \geq 0$ such that the map $H^0(X, \mathcal{O}_X(nD)) \to H^0\left(X^{\text{\emph{red}}}, \mathcal{O}_{X^{\emph{red}}}(\iota^*(nD))\right)$ is surjective for all $n \geq n_0$.
\end{lemma}
\begin{proof}
Let $\mathcal{N} \subset \mathcal{O}_X$ be the coherent sheaf of ideals that cuts out $X^{\text{red}}$ as a closed subscheme of~$X$. Since $D$ is an ample divisor on~$X$, there is some $n_0 \geq 0$ such that $H^1\left(X, \mathcal{N} \otimes_{\mathcal{O}_X} \mathcal{O}_X(nD)\right) = 0$ for all $n \geq n_0$ \cite[\href{https://stacks.math.columbia.edu/tag/0B5U}{Tag 0B5U}]{stacks-project}.

For all $n \geq n_0$ the long exact cohomology sequence arising from the short exact sequence of sheaves on $X$, $0 \to \mathcal{N} \otimes_{\mathcal{O}_X} \mathcal{O}_X(nD)\to \mathcal{O}_X(nD) \to \mathcal{O}_X / \mathcal{N} \otimes_{\mathcal{O}_X} \mathcal{O}_X(nD) \to 0$, yields the 3-term short exact sequence
\[
	H^0(X, \mathcal{O}_X(nD)) \to H^0(X, \mathcal{O}_X / \mathcal{N} \otimes_{\mathcal{O}_X} \mathcal{O}_X(nD)) \to H^1(X, \mathcal{N} \otimes_{\mathcal{O}_X} \mathcal{O}_X(nD)) = 0.
\]
Thus the map $H^0(X, \mathcal{O}_X(nD)) \to H^0(X, \mathcal{O}_X / \mathcal{N} \otimes_{\mathcal{O}_X} \mathcal{O}_X(nD))$ is surjective. Moreover, since $\iota_*\mathcal{O}_{X^{\text{red}}} \cong \mathcal{O}_X / \mathcal{N}$ as sheaves on $X$, we have
\[
	\mathcal{O}_X / \mathcal{N} \otimes_{\mathcal{O}_X} \mathcal{O}_X(nD) \cong \iota_*\mathcal{O}_{X^{\text{red}}} \otimes_{\mathcal{O}_X} \mathcal{O}_X(nD) \cong \iota_*\left(\mathcal{O}_{X^{\text{red}}} \otimes_{\mathcal{O}_{X^{\text{red}}}} \iota^*\mathcal{O}_X(nD)\right) \cong \iota_*\iota^* \mathcal{O}_X(nD).
\]
Here, the second isomorphism follows from the Projection Formula. Next, since closed immersions are affine \cite[\href{https://stacks.math.columbia.edu/tag/01SE}{Tag 01SE}]{stacks-project}, we have $H^0(X, \iota_*\iota^*\mathcal{O}_X(nD)) = H^0\left(X^{\text{red}}, \iota^*\mathcal{O}_X(nD)\right)$ \cite[\href{https://stacks.math.columbia.edu/tag/089W}{Tag 089W}]{stacks-project}. Furthermore, $\iota^*\mathcal{O}_X(nD) \cong \mathcal{O}_{X^{\text{red}}}(\iota^*(nD))$ \cite[p.~312]{gw}, thus $H^0\left(X^{\text{red}}, \iota^*\mathcal{O}_X(nD)\right) \cong H^0\left(X^{\text{red}}, \mathcal{O}_{X^{\text{red}}}(\iota^*(nD))\right)$. Therefore
\[
	H^0(X, \mathcal{O}_X / \mathcal{N} \otimes_{\mathcal{O}_X} \mathcal{O}_X(nD)) \cong H^0(X, \iota_*\iota^*\mathcal{O}_X(nD)) \cong H^0\left(X^{\text{red}}, \mathcal{O}_{X^{\text{red}}}(\iota^*(nD))\right).
\]
Finally, since the map $H^0(X, \mathcal{O}_X(nD)) \to H^0(X, \mathcal{O}_X / \mathcal{N} \otimes_{\mathcal{O}_X} \mathcal{O}_X(nD))$ is surjective, by the isomorphisms above, the map $H^0(X, \mathcal{O}_X(nD)) \to H^0\left(X^{\text{red}}, \mathcal{O}_{X^{\text{red}}}(\iota^*(nD))\right)$ is surjective.
\end{proof}

\begin{cor}
\label{surjective map on global sections}
Let $T$ be a complete discrete valuation ring with residue field $k$, and let $\mathscr{X}$ be a regular connected projective $T$-curve with closed fiber $X$. Let $X^{\text{\emph{red}}}$ be the reduced closed fiber of~$\mathscr{X}$, and let $\iota: X^{\text{\emph{red}}} \to X$ be the corresponding closed immersion. Let $S$ be a finite set of closed points of~$\mathscr{X}$ that meets each irreducible component of $X^{\text{\emph{red}}}$ non-trivially. Then there is an effective Cartier divisor $\widehat{D}$ on $\mathscr{X}$ whose support meets $X^{\text{\emph{red}}}$ precisely at $S$ and some integer $n_0 \geq 0$ such that the map $H^0(\mathscr{X}, \mathcal{O}_{\mathscr{X}} (n\widehat{D})) \to H^0(X^{\text{\emph{red}}}, \mathcal{O}_{X^{\emph{red}}}(\iota^*(nD)))$ is surjective for all $n \geq n_0$, where $D$ is the restriction of $\widehat{D}$ to $X$. Furthermore, the divisor $\iota^*D$ on~$X^{\text{\emph{red}}}$ has support $S$.
\end{cor}
\begin{proof}
By \cite[Proof of Proposition~3.3]{hhk15}, for each $P \in S$ we can find an effective Cartier divisor~$\widehat{D}_P$ on~$\mathscr{X}$ whose support meets $X^{\text{red}}$ precisely at $P$. Since $S$ meets each irreducible component of $X^{\text{red}}$ non-trivially, if we let $\widehat{D} = \sum_{P \in S} \widehat{D}_P$, then $\widehat{D}$ is an effective Cartier divisor on $\mathscr{X}$ whose restriction to $X$ is an effective divisor $D$ whose support meets each irreducible component of~$X$ non-trivially. Moreover, by construction, $\iota^*D$ has support $S$. We also observe that, by the first paragraph of \cite[Proof of Proposition~3.3]{hhk15}, the effective divisor $D$ on $X$ is ample. 

Since $D$ is an ample divisor on $X$, by Lemma \ref{surjective from closed fiber} there is an $n_1 \geq 0$ such that the map $H^0(X, \mathcal{O}_X(nD)) \to H^0(X^{\text{red}}, \mathcal{O}_{X^{\text{red}}}(\iota^*(nD)))$ is surjective for all $n \geq n_1$. Furthermore, since $\mathscr{X}$ is a projective $T$-curve and $D$ is ample on $X$, by \cite[\href{https://stacks.math.columbia.edu/tag/0D2M}{Proof of Tag 0D2M}]{stacks-project} there is an $n_2 \geq 0$ such that the map $H^0(\mathscr{X}, \mathcal{O}_\mathscr{X}(n\widehat{D})) \to H^0(X, \mathcal{O}_X(nD))$ is surjective for all $n \geq n_2$. So if we let $n_0 = \max\{n_1, n_2\} \geq 0$, then the composition of these surjective maps yields the surjective map
\[
	H^0(\mathscr{X}, \mathcal{O}_\mathscr{X}(n\widehat{D})) \to H^0(X^{\text{red}}, \mathcal{O}_{X^{\text{red}}}(\iota^*(nD)))
\]
for all $n \geq n_0$, as desired.
\end{proof}
The next two results will be used to check that the lift of our quadratic form has certain desired isotropy properties.
\begin{lemma}
\label{element that reduces to uniformizer}
Let $T$ be a complete discrete valuation ring with residue field $k$, and let $\mathscr{X}$ be a regular connected projective $T$-curve whose reduced closed fiber $X^{\text{\emph{red}}}$ is a union of regular irreducible $k$-curves, with normal crossings. Then for each closed point $P$ of $\mathscr{X}$ there is an element $r_P \in \widehat{R}_P$ such that, for each irreducible component $X_i$ of $X^{\emph{red}}$ that contains $P$, $r_P$ has non-zero image $\overline{r}_P$ in~$\widehat{\mathcal{O}}_{X_i, P}$ and~$\overline{r}_P$ generates the maximal ideal of $\widehat{\mathcal{O}}_{X_i, P}$.
\end{lemma}
\begin{proof}
We first observe that since $\mathscr{X}$ is regular, for all closed points $P$ of $\mathscr{X}$ the ring $\widehat{R}_P$ is a regular local ring of dimension two. In particular, since $\widehat{R}_P$ is a regular local ring, it is a unique factorization domain \cite[Theorem~5]{auslander}, so every height one prime ideal of $\widehat{R}_P$ is principal.

Now let $P$ be an arbitrary closed point of $\mathscr{X}$. By our assumption that $X^{\text{red}}$ is a union of regular irreducible $k$-curves with normal crossings, the point $P$ belongs to either one or two irreducible components of~$X^{\text{red}}$.

First suppose that $P$ belongs to a unique irreducible component $X_i$ of $X^{\text{red}}$ with generic point~$\eta_i$. The height one prime $\mathfrak{p}$ defining $\eta_i$ in $\widehat{R}_P$ is principal, so $\mathfrak{p} = (f_P)$ for some some $f_P \in \widehat{R}_P$. Now, because $\widehat{R}_P$ is a two-dimensional regular local ring, there must be another element $g_P \in \widehat{R}_P$ such that $g_P \not\in \mathfrak{p}$ and $\{f_P, g_P\}$ is a regular system of parameters of $\widehat{R}_P$. Since $\widehat{R}_P / \mathfrak{p} \cong \widehat{\mathcal{O}}_{X_i, P}$, the non-zero image $\overline{g}_P$ of $g_P$ in $\widehat{\mathcal{O}}_{X_i, P}$ generates the maximal ideal of $\widehat{\mathcal{O}}_{X_i, P}$. Letting $r_P = g_P$ proves the lemma in this case.

Now suppose that $P$ belongs to two irreducible components $X_i$ and $X_j$ of $X^{\text{red}}$, with generic points $\eta_i$ and $\eta_j$, respectively. If $\mathfrak{p}_i$ and $\mathfrak{p}_j$ are the height one primes that define $\eta_i$ and $\eta_j$ in~$\widehat{R}_P$, respectively, then we can find elements $f_P, g_P \in \widehat{R}_P$ such that $\mathfrak{p}_i = (f_P)$ and $\mathfrak{p}_j = (g_P)$. The points~$\eta_i$ and $\eta_j$ are distinct, so $\mathfrak{p}_i \ne \mathfrak{p}_j$. Therefore $f_P \not\in \mathfrak{p}_j$ and $g_P \not\in \mathfrak{p}_i$, so $\{f_P, g_P\}$ is a regular system of parameters of $\widehat{R}_P$. Moreover, the non-zero image of~$f_P$ in $\widehat{R}_P / \mathfrak{p}_j \cong \widehat{\mathcal{O}}_{X_j, P}$ generates the maximal ideal of $\widehat{\mathcal{O}}_{X_j, P}$, and the non-zero image of~$g_P$ in $\widehat{R}_P / \mathfrak{p}_i \cong \widehat{\mathcal{O}}_{X_i, P}$ generates the maximal ideal of $\widehat{\mathcal{O}}_{X_i, P}$. Thus $r_P = f_P + g_P$ is our desired element of $\widehat{R}_P$.
\end{proof}

The next lemma is analogous to Springer's Theorem, and gives a criterion for isotropy of a quadratic form over the fraction field of a complete local ring of any Krull dimension.
\begin{lemma}
\label{Hensel's Lemma for q forms}
Let $R$ be a complete local ring with residue field $\kappa$ of characteristic $\ne 2$ and fraction field $K$. For any integer $n \geq 2$ let $a_1, \ldots, a_n \in R^{\times}$ be units in $R$ with images $\overline{a}_1, \ldots, \overline{a}_n \in \kappa^{\times}$, and consider the quadratic form $q = \langle a_1, \ldots, a_n \rangle$ over $K$. If the quadratic form $\overline{q} = \left\langle \overline{a}_1, \ldots, \overline{a}_n \right\rangle$ is isotropic over $\kappa$, then $q$ is isotropic over $K$.
\end{lemma}
\begin{proof}
Since $\overline{q}$ is isotropic over $\kappa$, there exist elements $\overline{b}_1, \ldots, \overline{b}_n \in \kappa$, not all zero, such that
\[
	\overline{q} \left(\overline{b}_1, \ldots, \overline{b}_n\right) = \overline{a}_1\overline{b}_1^2 + \cdots + \overline{a}_n \overline{b}_n^2 = 0.
\]
Potentially after renumbering, we may assume that $\overline{b}_1, \ldots, \overline{b}_r \ne 0$ for some $2 \leq r \leq n$ and $\overline{b}_{r+1}, \ldots, \overline{b}_n = 0$.
Let $b_2, \ldots, b_r \in R^{\times}$ be unit lifts of $\overline{b}_2, \ldots, \overline{b}_r$, and consider the polynomial $f(x) = q(x, b_2, \ldots, b_r, 0, \ldots, 0) = a_1x^2 + a_2b_2^2 + \cdots + a_rb_r^2 \in R[x]$. By Hensel's Lemma, there is a unit lift $b_1 \in R^{\times}$ of $\overline{b}_1$ such that $f(b_1) = q(b_1, b_2, \ldots, b_r, 0, \ldots, 0) = 0$. Therefore $q$ is isotropic over~$K$, as desired.
\end{proof}
We are now ready to prove
\begin{prop}
\label{lifting universal forms on nice closed fiber}
Let $T$ be a complete discrete valuation ring with residue field $k$ of characteristic $\ne 2$ such that $m_s(k) = u_s(k) = 2^n$ for some integer $n \geq 1$. Let $\mathscr{X}$ be a regular connected projective $T$-curve whose reduced closed fiber $X^{\text{\emph{red}}}$ is a union of regular irreducible $k$-curves $X_1, \ldots, X_s$, with normal crossings. For $1 \leq i \leq s$ let $\eta_i$ be the generic point of $X_i$, and suppose that there is a universal quadratic form $q_1$ over $\kappa(\eta_1)$ such that $2 \leq d = \dim q < 2^{n+1}$. Let $F$ be the function field of $\mathscr{X}$. Then there is a $d$-dimensional quadratic form $\widetilde{q}$ over $F$ such that
\begin{enumerate}[label=(\alph*)]
	\item the first residue form of $\widetilde{q}_{\eta_1}$ over $F_{\eta_1}$ is isometric to $q_1$ over $\kappa(\eta_1)$ and the second residue form of $\widetilde{q}_{\eta_1}$ is trivial, and
	
	\item $\widetilde{q}$ is isotropic over $F_P$ for all points (not necessarily closed) $P \in X^{\text{\emph{red}}} \setminus \{\eta_1\}$.
\end{enumerate}
\end{prop}
\begin{proof}
Since $q_1$ is universal over $\kappa(\eta_1)$, it represents $-1$ over $\kappa(\eta_1)$, so $q_1 \simeq \langle -1, a_{1,2}, \ldots, a_{1,d} \rangle$ for some $a_{1,2}, \ldots, a_{1,d} \in \kappa(\eta_1)^{\times}$. Furthermore, since $X_1$ is a regular curve over $k$, for all closed points $P \in X_1$ the local ring $\mathcal{O}_{X_1, P}$ is a discrete valuation ring with associated discrete valuation $v_P$ on~$\kappa(\eta_1)$. In particular, if $s \geq 2$ and we let $\mathcal{P}_1$ be the non-empty finite set of closed points $P \in X_1$ at which $X_1$ intersects another component of $X^{\text{red}}$, then after scaling the entries of $q_1$ by non-zero squares in $\kappa(\eta_1)$, we may assume that $v_P(a_{1,r}) \in \{0, 1\}$ for all $P \in \mathcal{P}_1$ and all $r = 2, \ldots, d$ (see Remark~\ref{form with small valuation}). Moreover, by applying Lemma \ref{universal form on component of closed fiber} to the curve $X_1$, we conclude that $q_1$ is isotropic over $\kappa(\eta_1)_{v_P}$ for all closed points $P \in X_1$ (in particular, for all closed points $P \in \mathcal{P}_1$ if $s \geq 2$).

If $s \geq 2$, then by applying Proposition \ref{defining functions on components} to $X^{\text{red}}$, for each $i = 2, \ldots, s$ we can find elements $a_{i,2}, \ldots, a_{i,d} \in \kappa(\eta_i)^{\times}$ such that the quadratic form $q_i = \langle -1, a_{i,2}, \ldots, a_{i,d} \rangle$ is isotropic over $\kappa(\eta_i)$ and for all closed points $P \in X_i \cap X_j$ for some $i \ne j$, we have $\overline{a}_{i,r,P} = \overline{a}_{j,r,P} \in \kappa(P)$ for all $r = 2, \ldots, d$. For each $r = 2, \ldots, d$, if we let $a_r$ be defined by $a_{i,r}$ on $X_i$, then $a_r$ is a well-defined rational function on $X^{\text{red}}$ since the $a_{i,r}$ agree at intersection points.

Let $S$ be a non-empty finite set of closed points of $\mathscr{X}$ that meets each irreducible component of~$X^{\text{red}}$ non-trivially and contains all points $P$ such that $v_P(a_{i,r}) \ne 0$ for some $i \in \{1, \ldots, s\}$ and some $r \in \{2, \ldots, d\}$. By Corollary \ref{surjective map on global sections}, there is an effective Cartier divisor $\widehat{D}$ on $\mathscr{X}$ whose support meets $X^{\text{red}}$ precisely at $S$, with restriction~$D$ to the closed fiber $X$ of $\mathscr{X}$ such that the divisor $\iota^*D$ on $X^{\text{red}}$ has support $S$ for the closed immersion $\iota: X^{\text{red}} \to X$. Replacing $\widehat{D}$ (and consequently $D$ and~$\iota^*D$) with a sufficiently large multiple, we have $a_r \in H^0(X^{\text{red}}, \mathcal{O}_{X^{\text{red}}}(\iota^*D))$ for all $r = 2, \ldots, d$ by our choice of~$S$, and the map $H^0(\mathscr{X}, \mathcal{O}_{\mathscr{X}}(\widehat{D})) \to H^0(X^{\text{red}}, \mathcal{O}_{X^{\text{red}}}(\iota^*D))$ is surjective by Corollary~\ref{surjective map on global sections}. Therefore there are elements $\widetilde{a}_2, \ldots, \widetilde{a}_d \in F^{\times}$ such that $\widetilde{a}_r \in \mathcal{O}_{\mathscr{X}, \eta_i}^{\times}$ with $\overline{\widetilde{a}}_{r, \eta_i} = a_{i,r} \in \kappa(\eta_i)^{\times}$ for all $i = 1, \ldots, s$ and all $r = 2, \ldots, d$.

Now let $\widetilde{q}$ be the quadratic form defined over $F$ by $\widetilde{q} = \left\langle -1, \widetilde{a}_2, \ldots, \widetilde{a}_d \right\rangle$. For each $i = 1, \ldots, s$, if we let $\widetilde{q}_{\eta_i}$ denote $\widetilde{q}$ considered as a form over $F_{\eta_i}$, then the first residue form of $\widetilde{q}_{\eta_i}$ is isometric to $q_i$ over $\kappa(\eta_i)$, and the second residue form of~$\widetilde{q}_{\eta_i}$ is trivial, proving (a). Moreover, because $q_i$ is isotropic over $\kappa(\eta_i)$ for all $i = 2, \ldots, s$, Springer's Theorem implies that $\widetilde{q}$ is isotropic over~$F_{\eta_i}$ for all $i = 2, \ldots, s$. This proves part (b) for non-closed points $P \in X^{\text{red}} \setminus \{\eta_1\}$.

To complete the proof, we must show that $\widetilde{q}$ is isotropic over $F_P$ for all closed points $P \in$~$X^{\text{red}}$. Let $P \in X^{\text{red}}$ be an arbitrary closed point. Then $P \in X_i$ for some $i \in \{1, \ldots, s\}$. The ring $\mathcal{O}_{X_i,P}$ is a discrete valuation ring with associated discrete valuation $v_P$ on $\kappa(\eta_i)$. By Lemma \ref{element that reduces to uniformizer}, there is an element $r_P \in \widehat{R}_P$ whose non-zero image $\overline{r}_P$ in $\widehat{\mathcal{O}}_{X_i, P}$ is a uniformizer for $v_P$.  So by multiplying and dividing the entries of~$\widetilde{q}$ by even powers of $r_P$, over $F_P$ we can write
\[
	\widetilde{q} \simeq \widetilde{q}_P = \widetilde{q}_{1,P} \perp r_P \cdot \widetilde{q}_{2,P},
\]
where the entries of $\widetilde{q}_{1,P}, \widetilde{q}_{2,P}$ all belong to $\widehat{R}_P^{\times}$. Reducing modulo~$\eta_i$, over $\kappa(\eta_i)_{v_P}$ we have
\[
	q_i \simeq q_P = q_{1,P} \perp \overline{r}_P \cdot q_{2,P},
\]
where the entries of $q_{1,P}, q_{2,P}$ are all units in $\widehat{\mathcal{O}}_{X_i,P}$. The quadratic form $q_i$ is isotropic over $\kappa(\eta_i)_{v_P}$, thus so is $q_P$. So by Springer's Theorem, at least one of the residue forms $\overline{q}_{1,P}$ or $\overline{q}_{2,P}$ of $q_P$ must be isotropic over $\kappa(P)$. The entries of $\overline{q}_{1,P}, \overline{q}_{2,P}$ are the images of the entries of $\widetilde{q}_{1,P},  \widetilde{q}_{2,P}$, respectively, in $\kappa(P)^{\times}$. So by Lemma~\ref{Hensel's Lemma for q forms}, we conclude that one of the quadratic forms $\widetilde{q}_{1,P}$ or~$\widetilde{q}_{2,P}$ is isotropic over $F_P$, and therefore $\widetilde{q}_P \simeq \widetilde{q}$ is isotropic over $F_P$ as well. 
\end{proof} 

\begin{prop}
\label{au contains union}
Let $T$ be a complete discrete valuation ring with fraction field $K$ and residue field $k$ of characteristic $\ne 2$ such that $m_s(k) = u_s(k) < \infty$. Let $\mathscr{X}$ be a regular connected projective $T$-curve whose closed fiber $X$ has irreducible components $X_1, \ldots, X_s$. For $1 \leq i \leq s$ let $\eta_i$ be the generic point of~$X_i$, and let $F$ be the function field of $\mathscr{X}$. Then
\[
	\bigcup_{i = 1}^s \left\{r_1 + r_2 \mid r_1, r_2 \in \emph{AU}(\kappa(\eta_i))\right\} \subseteq \emph{AU}(F).
\]
\end{prop}
\begin{proof}
We first observe that by taking a suitable blow-up $\mathscr{X}' \to \mathscr{X}$ we obtain a regular connected projective $T$-curve~$\mathscr{X}'$ with function field $F$ whose reduced closed fiber $X'^{\text{red}}$ is a union of $s' \geq s$ regular irreducible curves $X_1', \ldots, X_{s'}'$ over $k$, with normal crossings. For $1 \leq j \leq s'$ let $\xi_j$ be the generic point of $X_j'$. Since $\mathscr{X}'$ is a blow-up of the regular $T$-curve $\mathscr{X}$, for each $i \in \{1, \ldots, s\}$ there is some $j \in$~$\{1, \ldots, s'\}$ such that the local rings $\mathcal{O}_{\mathscr{X}, \eta_i}$ and  $\mathcal{O}_{\mathscr{X}', \xi_j}$ are isomorphic. Therefore $\kappa(\eta_i) \cong \kappa(\xi_j)$, hence $\au(\kappa(\eta_i)) = \au(\kappa(\xi_j))$. The proposition will therefore be proven if we show that $\bigcup_{j = 1}^{s'} \{r_1 + r_2 \mid r_1, r_2 \in \au(\kappa(\xi_j))\} \subseteq \au(F)$.

To that end, we first note that since $m_s(k) = u_s(k) < \infty$, we have $m_s(k) = u_s(k) = 2^n$ for some integer $n \geq 0$ by Corollary \ref{if strong m = strong u}. Hence $u(\kappa(\xi_j)) = 2^{n+1} \in \au(\kappa(\xi_j))$ for all $j = 1, \ldots, s'$ by Lemma~\ref{u-inv of residue fields}. Moreover, since $m_s(k) = u_s(k) = 2^n$, we have $m_s(K) = u_s(K) = 2^{n+1}$ by Theorem \ref{strong m of cdvf}, so $u(F) = 2^{n+2} = 2^{n+1} + 2^{n+1} \in \au(F)$ by Proposition \ref{u-inv if m_s = u_s}. Therefore for all $j \in \{1, \ldots, s'\}$ such that $m(\kappa(\xi_j)) = u(\kappa(\xi_j)) = 2^{n+1}$ we have $\{r_1 + r_2 \mid r_1, r_2 \in \au(\kappa(\xi_j))\} = \left\{2^{n+2}\right\} \subseteq \au(F)$. This completes the proof if $m_s(k) = u_s(k) = 1$. Indeed, by Lemma \ref{m at least 2 for fg trdeg 1} we have $m(\kappa(\xi_j)) \geq 2$ for all $j = 1, \ldots, s'$, hence $m(\kappa(\xi_j)) = u(\kappa(\xi_j)) = 2$ for all $j = 1, \ldots, s'$.

To complete the proof, we must consider the case when $m_s(k) = u_s(k) = 2^n$ for $n \geq 1$ and there is at least one $j \in \{1, \ldots, s'\}$ such that $2 \leq m(\kappa(\xi_j)) < u(\kappa(\xi_j)) = 2^{n+1}$. Suppose this is the case. After renumbering the components if necessary, we may assume $2 \leq m(\kappa(\xi_1)) < u(\kappa(\xi_1)) = 2^{n+1}$. We need to show that $r_1 + r_2 \in \au(F)$ for all $r_1, r_2 \in \au(\kappa(\xi_1))$. If $r_1 = r_2 = u(\kappa(\xi_1)) = 2^{n+1}$, then we have already seen that $r_1 + r_2 = 2^{n+2} = u(F) \in \au(F)$. So suppose we are given $r_1, r_2 \in \au(\kappa(\xi_1))$ with at least one of $r_1, r_2$ less than $2^{n+1}$. Without loss of generality, we may assume $2 \leq r_1 < 2^{n+1}$. Let $q_1$ be an $r_1$-dimensional anisotropic universal quadratic form over $\kappa(\xi_1)$. Then by Proposition \ref{lifting universal forms on nice closed fiber} there is an $r_1$-dimensional quadratic form $\widetilde{q}_1$ over $F$ such that the first residue form of $\widetilde{q}_{1,\xi_1}$ over $F_{\xi_1}$ is isometric to $q_1$ over $\kappa(\xi_1)$, the second residue form of $\widetilde{q}_{1,\xi_1}$ is trivial, and $\widetilde{q}_1$ is isotropic over~$F_P$ for all points $P \in X'^{\text{red}} \setminus \{\xi_1\}$. 

Next, let $q_2 = \langle b_1, \ldots, b_{r_2} \rangle$ be any $r_2$-dimensional anisotropic universal quadratic form over~$\kappa(\xi_1)$, and for unit lifts $\widetilde{b}_1, \ldots, \widetilde{b}_{r_2} \in \mathcal{O}_{\mathscr{X}', \xi_1}^{\times}$ of $b_1, \ldots, b_{r_2}$ to $F$, let $\widetilde{q}_2 = \langle \widetilde{b}_1, \ldots, \widetilde{b}_{r_2} \rangle$. Let $\pi_{\xi_1}$ be a uniformizer for the discrete valuation $v_{\xi_1}$ on $F$, and consider the $(r_1 + r_2)$-dimensional quadratic form $\widetilde{\varphi}$ over $F$ defined by
\[
	\widetilde{\varphi} = \widetilde{q}_1 \perp \pi_{\xi_1} \cdot \widetilde{q}_2.
\]
Over $F_{\xi_1}$, the first and second residue forms of $\widetilde{\varphi}_{\xi_1}$ are isometric to the anisotropic universal forms~$q_1$ and $q_2$, respectively, over $\kappa(\xi_1)$. Therefore $\widetilde{\varphi}_{\xi_1}$ is anisotropic and universal over $F_{\xi_1}$ by Lemma \ref{anisotropic universal residue forms}. For all points $P \in X'^{\text{red}} \setminus \{\xi_1\}$, since $\widetilde{q}_1$ is isotropic over $F_P$, so is $\widetilde{\varphi}$. So $\widetilde{\varphi}$ is universal over~$F_Q$ for all points $Q \in X'^{\text{red}}$ with $\dim \widetilde{\varphi} = r_1 + r_2 \geq 4$, hence $\widetilde{\varphi}$ is universal over $F$ by Lemma \ref{everywhere locally universal implies globally universal} and \cite[Theorem~9.3]{hhk15}. Furthermore, $\widetilde{\varphi}$ is anisotropic over $F$ since it is anisotropic over $F_{\xi_1}$. So~$\widetilde{\varphi}$ is anisotropic and universal over $F$ with $\dim \widetilde{\varphi} = r_1 + r_2$. This construction can be done for any $r_1, r_2 \in \au(\kappa(\xi_1))$ with $2 \leq r_1 < 2^{n+1}$, so we have proven
\[
	\{r_1 + r_2 \mid r_1, r_2 \in \au(\kappa(\xi_1))\} \subseteq \au(F).
\]
The above argument can be applied to any $\kappa(\xi_j)$ such that $2 \leq m(\kappa(\xi_j)) < u(\kappa(\xi_j)) = 2^{n+1}$, so the proof of the proposition is complete.
\end{proof}

We can now prove
\begin{theorem}
\label{au equals union}
Let $T$ be a complete discrete valuation ring with residue field $k$ of characteristic $\ne 2$ such that $m_s(k) = u_s(k) < \infty$. Let $\mathscr{X}$ be a regular connected projective $T$-curve with closed fiber~$X$. Let $X_1, \ldots, X_s$ be the irreducible components of $X$, and for $1 \leq i \leq s$ let $\eta_i$ be the generic point of $X_i$. Let $\Gamma$ be the reduction graph of $\mathscr{X}$, and let $F$ be the function field of $\mathscr{X}$. Then
\[
	\emph{AU}(F) = \begin{cases}
	\{2\} \cup \bigcup_{i = 1}^s \left\{r_1 + r_2 \mid r_1, r_2 \in \emph{AU}(\kappa(\eta_i))\right\} &\text{ if $\Gamma$ is not a tree}, \\
		\bigcup_{i = 1}^s \left\{r_1 + r_2 \mid r_1, r_2 \in \emph{AU}(\kappa(\eta_i))\right\} &\text{ if $\Gamma$ is a tree}.
	\end{cases}
\]
\end{theorem}
\begin{proof}
The containment $\subseteq$ follows from Proposition \ref{au contained in union}. Next, we know that $2 \in \au(F)$ if and only if $m(F) = 2$, and by Proposition \ref{m=2 iff loop}, $m(F) = 2$ if and only if $\Gamma$ is not a tree. Finally, by Proposition~\ref{au contains union} we have $\bigcup_{i = 1}^s \left\{r_1 + r_2 \mid r_1, r_2 \in \au(\kappa(\eta_i))\right\} \subseteq \au(F)$. We have both containments, which proves the desired equality.
\end{proof}

Recall that a semi-global field has $m$-invariant 2 if and only if the reduction graph of a regular model is not a tree (see Proposition \ref{m=2 iff loop}). We finish this section by showing how Theorem \ref{au equals union} can be used to calculate the $m$-invariant of a semi-global field with a regular model whose reduction graph is a tree.

\begin{cor}
\label{m-inv of sg field is min of m-inv of residue fields}
Let $T$ be a complete discrete valuation ring with residue field $k$ of characteristic $\ne 2$ such that $m_s(k) = u_s(k) < \infty$. Let $\mathscr{X}$ be a regular connected projective $T$-curve with closed fiber $X$. Let $X_1, \ldots, X_s$ be the irreducible components of $X$, and for $1 \leq i \leq s$ let $\eta_i$ be the generic point of $X_i$. Assume the reduction graph of $\mathscr{X}$ is a tree, let $F$ be the function field of $\mathscr{X}$, and let $m^* = \min_{i = 1, \ldots, s} \left\{m(\kappa(\eta_i))\right\}$. Then $m(F) = 2m^*$.
\end{cor}
\begin{proof}
Let $q$ be any anisotropic universal quadratic form over $F$. Since $m_s(k) = u_s(k) < \infty$ and the reduction graph of $\mathscr{X}$ is a tree, Theorem \ref{au equals union} implies that
\[
	\dim q \in \bigcup_{i = 1}^s \left\{r_1 + r_2 \mid r_1, r_2 \in \au(\kappa(\eta_i))\right\}.
\]
For each $i = 1, \ldots, s$ and $r_1, r_2 \in \au(\kappa(\eta_i))$ we have $r_1 + r_2 \geq 2m(\kappa(\eta_i))$ since $r_1, r_2 \geq m(\kappa(\eta_i))$. So by our definition of $m^*$, we have
$\dim q \geq 2m^*$. Since this holds for any anisotropic universal quadratic form over $F$, it follows that $m(F) \geq 2m^*$. 

To prove the opposite inequality, let $i^* \in \{1, \ldots, s\}$ be such that $m(\kappa(\eta_{i^*})) = m^*$. By Theorem~\ref{au equals union}, there exists a $2m^*$-dimensional anisotropic universal quadratic form over $F$. Therefore $m(F) \leq 2m^*$. We have both inequalities, hence $m(F) = 2m^*$. 
\end{proof}

\section{Semi-global fields over $n$-local fields}
\label{classifying curves}
Let $k$ be a field and let $n \geq 1$ be a positive integer. Recall that a complete discretely valued field~$K$ is called an \textit{$n$-local field over $k$} if $K$ fits into a chain of fields $K = K_n, K_{n-1}, \ldots, K_1, K_0 = k$, where, for $i = 1, \ldots, n$, $K_i$ is a complete discretely valued field with residue field $K_{i-1}$. When we refer to the \textit{residue field} of $K$, we mean the field $K_{n-1}$. In this section, we define the \textit{layer} (see Definition~\ref{layer defin}) of a semi-global field $F$ over such a field $K$ in order to calculate $m(F)$. Throughout most of this section we will assume that the field $k$ has characteristic $\ne 2$ and satisfies $m_s(k) = u_s(k) < \infty$. This is a rather mild condition that is satisfied by, e.g., algebraically closed fields, finite fields, $n$-local fields over finite fields, and finitely generated transcendence degree one extensions of algebraically closed fields (see Examples \ref{strong m examples}(3)).

Let $k$ be a field and for any integer $n \geq 2$ let $K_n$ be an $n$-local field over $k$ with valuation ring~$T_n$ whose residue field $K_{n-1}$ has valuation ring $T_{n-1}$. Let $F$ be a semi-global field over $K_n$, and let~$\mathscr{X}$ be a regular model of $F$ over $T_n$ with reduction graph $\Gamma_{\mathscr{X}}$ and closed fiber $X$. Let $X_1, \ldots, X_s$ be the irreducible components of $X$, and for $1 \leq i \leq s$ let $\eta_i$ be the generic point of $X_i$. For each $i = 1, \ldots, s$ the field~$\kappa(\eta_i)$ is a finitely generated transcendence degree one extension of the $(n-1)$-local field $K_{n-1}$ over $k$, hence $\kappa(\eta_i)$ is a semi-global field over $K_{n-1}$. So we can find a regular model $\mathscr{Y}_i$ of $\kappa(\eta_i)$ over $T_{n-1}$. Each curve~$\mathscr{Y}_i$ has a reduction graph $\Gamma_{\mathscr{Y}_i}$ and a closed fiber~$Y_i$ with irreducible components, etc. The layer of $F$ captures the first step in this process in which we find a reduction graph that is not a tree.

To define the layer of a semi-global field, we use rooted trees. Recall that a \textit{rooted tree} is a tree in which one of the vertices has been designated the root. For each vertex $v$ in a rooted tree, there is a unique path from the root to $v$, and the \textit{level} of $v$ is the number of vertices in the path from the root to $v$, including both $v$ and the root. For example, the root of a rooted tree has level 1. A \textit{child} of a vertex $v$ in a rooted tree is a vertex~$w$ connected to~$v$ by an edge such that level$(w) = \text{level}(v) + 1$. We call $v$ a \textit{parent} of~$w$. In a rooted tree, the root is the only vertex without a parent; all other vertices have a unique parent.

Given a semi-global field $F$ over an $n$-local field over a field $k$, we will now inductively define a procedure to associate to $F$ a rooted tree $\tau_{F}$ whose root vertex is labeled by $F$ and whose vertices are each colored either black or white.
\begin{defin}
\label{assoc. rooted tree}
Let $k$ be any field, let $K_1$ be any 1-local field over $k$, and let $F_1$ be a semi-global field over $K_1$. We let $\tau_{F_1}$ be the rooted tree with one vertex, $v_{F_1}$. If the reduction graph of a regular model of $F_1$ is a tree, then color $v_{F_1}$ white, and if the reduction graph of a regular model of~$F_1$ is not a tree, then color $v_{F_1}$ black.

Now for some integer $n \geq 1$ assume we have defined a procedure to associate to any semi-global field $E$ over any $n$-local field over $k$ a rooted tree $\tau_{E}$ with root vertex $v_E$ and whose vertices are each colored either black or white. Let $K_{n+1}$ be an $(n+1)$-local field over $k$ with residue field $K_n$, and let $F_{n+1}$ be a semi-global field over $K_{n+1}$. Let $\mathscr{X}$ be a regular model of $F_{n+1}$ with closed fiber~$X$ and reduction graph $\Gamma$. Let $X_1, \ldots, X_s$ be the irreducible components of $X$, and for $1 \leq i \leq s$ let~$\eta_i$ be the generic point of $X_i$. For each $1 \leq i \leq s$ the field $\kappa(\eta_i)$ is a semi-global field over $K_n$, so by the induction hypothesis we can associate to $\kappa(\eta_i)$ a rooted tree $\tau_{\kappa(\eta_i)}$ with root vertex $v_{\kappa(\eta_i)}$. We then let~$\tau_{F_{n+1}}$ be the rooted tree with root vertex $v_{F_{n+1}}$ obtained by adding an edge to each $\tau_{\kappa(\eta_i)}$ connecting the root~$v_{\kappa(\eta_i)}$ of~$\tau_{\kappa(\eta_i)}$ to the vertex $v_{F_{n+1}}$ so that $v_{\kappa(\eta_i)}$ is a child of~$v_{F_{n+1}}$. Color $v_{F_{n+1}}$ white if~$\Gamma$ is a tree, and color $v_{F_{n+1}}$ black if $\Gamma$ is not a tree. For all other vertices in~$\tau_{F_{n+1}}$ belonging to~$\tau_{\kappa(\eta_i)}$ for some $i \in \{1, \ldots, s\}$ we retain the color of that vertex as a vertex in $\tau_{\kappa(\eta_i)}$.
\end{defin}
\begin{remark}
\label{assoc. tree may not be unique}
Let $F$ be a semi-global field over an $n$-local field $K$ over a field $k$ for some integer $n \geq 1$. Then a regular model of $F$ always exists, but this model is not unique. However, if $\mathscr{X}_1, \mathscr{X}_2$ are regular models of $F$ with reduction graphs $\Gamma_1, \Gamma_2$, respectively, then $\Gamma_1$ is a tree if and only if $\Gamma_2$ is a tree \cite[Corollary~7.8]{hhk15}. So if $n = 1$, then the rooted tree $\tau_{F}$ associated to~$F$ via Definition \ref{assoc. rooted tree} is unique, and if $n \geq 2$, then $\tau_{F}$ is not unique and depends on our choice of regular model. We will call any rooted tree associated to $F$ via Definition \ref{assoc. rooted tree} a \textit{rooted component tree} of~$F$.
\end{remark}

Next, we define the layer of a rooted component tree.
\begin{defin}
\label{tree layer}
For any integer $n \geq 1$ let $K$ be an $n$-local field over a field $k$. Let $F$ be a one-variable function field over $K$, and let $\tau$ be a rooted component tree of $F$. We define the \textit{layer} of~$\tau$ by $\layer(\tau) = \inf\{\text{level}(v) \mid v \text{ a black vertex in $\tau$}\}$, where the infimum of the empty set is $\infty$.
\end{defin}

The following lemma will be used to prove Proposition \ref{layer is well-defined}, which states that if $F$ is a semi-global field over an $n$-local field over a field $k$ of characteristic $\ne 2$ such that $m_s(k) = u_s(k) < \infty$, then all rooted component trees of $F$ have the same layer.
\begin{lemma}
\label{identifying vertices}
Let $k$ be a field of characteristic $\ne 2$ such that $m_s(k) = u_s(k) < \infty$. For any integer $n \geq 1$ let $K$ be an $n$-local field over $k$, and let $F_1, F_2$ be one-variable function fields over $K$ with $m(F_1) = m(F_2) > 2$. For $i = 1, 2$, let $\mathscr{X}_i$ be any regular model of~$F_i$ with closed fiber $X_i$ whose irreducible components $X_{i, 1}, \ldots, X_{i, s_i}$ have generic points $\eta_{i, 1}, \ldots, \eta_{i, s_i}$, respectively. Then there exist $j_1 \in \{1, \ldots, s_1\}$ and $j_2 \in \{1, \ldots, s_2\}$ such that 
$m(\kappa(\eta_{1, j_1})) = m(\kappa(\eta_{2, j_2})) \leq m(\kappa(\eta_{i, r_i}))$ for all $i = 1, 2$ and all $r_i = 1, \ldots, s_i$.
\end{lemma}
\begin{proof}
Since $m_s(k) = u_s(k) < \infty$, the residue field $\kappa$ of $K$ satisfies $m_s(\kappa) = u_s(\kappa) < \infty$ by Theorem \ref{strong m of cdvf} and induction. Because $m(F_1) = m(F_2) > 2$, the reduction graphs of $\mathscr{X}_1$ and $\mathscr{X}_2$ are trees, so by applying Corollary~\ref{m-inv of sg field is min of m-inv of residue fields} to $\mathscr{X}_1, \mathscr{X}_2$ we have 
\[
	\min_{r_1 = 1, \ldots, s_1} \{2m(\kappa(\eta_{1, r_1}))\} = m(F_1) = m(F_2) = \min_{r_2 = 1, \ldots, s_2} \{2m(\kappa(\eta_{2, r_2}))\}.
\]
For $i = 1, 2$, letting $j_i \in \{1, \ldots, s_i\}$ be such that $m(\kappa(\eta_{i, j_i})) = \min_{r_i = 1, \ldots, s_i} \{m(\kappa(\eta_{i, r_i}))\}$ proves the lemma.
\end{proof}
\begin{prop}
\label{layer is well-defined}
Let $k$ be a field of characteristic $\ne 2$ such that $m_s(k) = u_s(k) < \infty$. For any integer $n \geq 1$ let $K$ be an $n$-local field over $k$, and let $F$ be a one-variable function field over $K$. Then for any rooted component trees $\tau_1, \tau_2$ of $F$ we have $\emph{layer}(\tau_1) = \emph{layer}(\tau_2)$.
\end{prop}
\begin{proof} 
By definition, $\tau_1$ and $\tau_2$ have the same root vertex, $v_F$, which is colored black if and only if the reduction graph of a regular model of $F$ is not a tree. So if the reduction graph of a regular model of $F$ is not a tree, then $\layer(\tau_1) = \layer(\tau_2) = 1$, and if the reduction graph of a regular model of $F$ is a tree, then $\layer(\tau_1), \layer(\tau_2) > 1$. This completes the proof if $n = 1$.

So to finish the proof we may assume that $n \geq 2$ and the reduction graph of a regular model of~$F$ is a tree. For each $j \in \{1, \ldots, n\}$ the level $j$ vertices of $\tau_1, \tau_2$ correspond to one-variable function fields~$E_{j}$ over the same $(n-j+1)$-local field over $k$. Moreover, $v_{E_{j}}$ is black if and only if $m(E_{j}) = 2$. So for each~$j$ such that $2 \leq j \leq \min\{n, \layer(\tau_1)\}$, by iteratively applying Lemma \ref{identifying vertices} and Corollary~\ref{m-inv of sg field is min of m-inv of residue fields} we can find level $j$ vertices $v_{E_{1, j}}, v_{E_{2, j}}$ of $\tau_1, \tau_2$, respectively, such that
$m(E_{1,j}) = m(E_{2,j}) \leq m(E_j)$ for all level $j$ vertices $v_{E_j}$ of $\tau_1$ and $\tau_2$, and for $i = 1, 2$, $v_{E_{i,j}}$ is a child of $v_{E_{i, j-1}}$ (here taking $E_{i,1} = F$ for $i = 1, 2$). 

For each $1 \leq j < \layer(\tau_1)$, we know that all level $j$ vertices of $\tau_1$ are white, thus the corresponding fields have $m$-invariant larger than 2. In particular, $m(E_{1, j}) > 2$ for all $1 \leq j < \layer(\tau_1)$. Then by our choice of the vertices $v_{E_{2, j}}$ in $\tau_2$, we have $m(E_{2, j}) > 2$ and all level $j$ vertices of $\tau_2$ are white for $1 \leq j < \layer(\tau_1)$. In particular, if $\layer(\tau_1) = \infty$, then $\layer(\tau_2) = \infty$.

Finally, if $2 \leq \layer(\tau_1) \leq n$, then $\tau_1$ has a black vertex of level $\layer(\tau_1)$. This black vertex corresponds to a field with $m$-invariant 2, so by our choice of $v_{E_{1, \layer(\tau_1)}}$, we have $m(E_{1, \layer(\tau_1)}) = 2$. This then implies that $m(E_{2, \layer(\tau_1)}) = 2$, so the vertex $v_{E_{2, \layer(\tau_1)}}$ is black. All vertices in $\tau_2$ of level less than $\layer(\tau_1)$ are white by the previous paragraph, so $\layer(\tau_2) = \layer(\tau_1)$, as desired.
\end{proof}

We now define the layer of a semi-global field over an $n$-local field over a field $k$ of characteristic $\ne 2$ such that $m_s(k) = u_s(k) < \infty$. 
\begin{defin}
\label{layer defin}
For any integer $n \geq 1$ let $K$ be an $n$-local field over a field $k$ of characteristic $\ne 2$ such that $m_s(k) = u_s(k) < \infty$, let $F$ be a one-variable function field over $K$, and let $\tau$ be any rooted component tree of $F$. We define the \textit{layer} of $F$ by $\layer(F) = \layer(\tau)$. If $\layer(F) = \infty$, then we say that $F$ is \textit{fully arboreal}.
\end{defin}
In the context of Definition \ref{layer defin} (i.e., when $\Char k \ne 2$ and $m_s(k) = u_s(k) < \infty)$, $\layer(F)$ is well-defined by Proposition \ref{layer is well-defined}. It would be interesting to know whether the layer is independent of the choice of rooted component tree even without these assumptions on $k$.

\begin{remark}
\label{layer in terms of components}
For any integer $n \geq 2$ let $K$ be an $n$-local field over a field $k$ of characteristic $\ne 2$ such that $m_s(k) = u_s(k) < \infty$, and let~$F$ be a one-variable function field over $K$. Let~$\mathscr{X}$ be any regular model of $F$ with closed fiber $X$ whose irreducible components $X_1, \ldots, X_s$ have generic points $\eta_1, \ldots, \eta_s$, respectively. If $F$ is fully arboreal, then~$\kappa(\eta_i)$ is fully arboreal for all $i = 1, \ldots, s$. If $F$ is not fully arboreal and $1 < \layer(F) \leq n$, then the set $I_0 = \{i \in \{1, \ldots, s\} \mid \kappa(\eta_i) \text{ not fully arboreal}\}$ is non-empty and $\layer(F) = \min_{i \in I_0} \{\layer(\kappa(\eta_i))\} + 1$.
\end{remark}

\begin{example}
\label{fully arboreal example}
Let $S$ be a complete discrete valuation ring with fraction field $L$ and residue field~$\ell$. The curve $\mathbb{P}^1_S$ is a regular model of $L(x)$ with closed fiber~$\mathbb{P}^1_{\ell}$, so the reduction graph of~$\mathbb{P}^1_S$ is trivial, hence a tree. Thus for any integer $n \geq 1$, if $K$ is an $n$-local field over a field $k$ of characteristic $\ne 2$ such that $m_s(k) = u_s(k) < \infty$, then $K(x)$ is fully arboreal.
\end{example}

\begin{prop}
\label{curves with layer j}
Let $k$ be a field of characteristic $\ne 2$ such that $m_s(k) = u_s(k) < \infty$, and let $K$ be an $n$-local field over $k$ for some integer $n \geq 1$. Then for each $j = 1, \ldots, n$ there is a semi-global field $F_j$ over $K$ with $\emph{layer}(F_j) = j$.
\end{prop}
\begin{proof}
First consider a complete discrete valuation ring $S$ with residue field $\ell$ of characteristic $\ne 2$, fraction field $L$, and uniformizer $\pi$. Then the field $L' = L(x)(\sqrt{x(x-\pi)(x-1)(x-1-\pi)})$ has a regular model~$\mathscr{X}$ whose closed fiber has two irreducible components, each isomorphic to $\mathbb{P}^1_{\ell}$, meeting at two points. The reduction graph of $\mathscr{X}$ is not a tree.

Now let $K = K_n, K_{n-1}, \ldots, K_1, K_0 = k$ be the chain of complete discretely valued fields into which $K$ fits as an $n$-local field over $k$, and for each $i = 1, \ldots, n$ let $T_i$ be the valuation ring of $K_i$ and let $\pi_i$ be a uniformizer for $K_i$. We consider two cases for $j \geq 1$.

First, if $j = 1$ and we let $F_1 = K(x)(\sqrt{x(x-\pi_n)(x-1)(x- 1 - \pi_n)})$, then by the first paragraph of the proof, $\layer(F_1) = 1$.

The remaining case considers the situation when $n \geq 2$ and $2 \leq j \leq n$. We will inductively define elements $a_{i,j} \in K_i^{\times}$ for each $n-j+1 \leq i \leq n$ to define $F_j$. Let $j \in \{2, \ldots, n\}$ be given, let $a_{n-j+1,j} = \pi_{n-j+1} \in K_{n-j+1}^{\times}$, and for each $n-j+2 \leq i \leq n$ let $a_{i,j} \in T_i^{\times}$ be a unit lift of $a_{i-1,j} \in K_{i-1}^{\times}$. For each $n-j+1 \leq i \leq n$ let $L_{i,j} = K_i(x)\left(\sqrt{x(x-a_{i,j})(x-1)(x-1-a_{i,j})}\right)$. By the first paragraph of the proof, $\layer(L_{n-j+1,j}) = 1$. For each $i \in \{n-j+2, \ldots, n\}$ the field $L_{i,j}$ has a smooth model over $T_i$ whose closed fiber has function field $L_{i-1,j}$. Since $\layer(L_{n-j+1,j}) = 1$, by Remark \ref{layer in terms of components} we have $\layer(L_{i,j}) = \layer(L_{i-1,j}) + 1 = i - n + j$. So if we let $F_j = L_{n,j}$, then $\layer(F_j) = j$. 
\end{proof}

For most of the results in the remainder of this section, we will consider $n$-local fields over a field~$k$ of characteristic $\ne 2$ that, in addition to satisfying $m_s(k) =$~$u_s(k) < \infty$, also satisfies the second, more restrictive condition that $m(L) = u(L) = 2u_s(k)$ for all finitely generated field extensions $L$ over $k$ of transcendence degree one (here noting that $u(L) = 2u_s(k)$ follows from Proposition \ref{u-inv if m_s = u_s}). Examples of fields satisfying both of these conditions are fields $k$ of characteristic $\ne 2$ such that $m_s(k) = u_s(k) = 1$ (e.g.,~$k$ is algebraically closed), and finite fields of odd characteristic. Indeed, if $m_s(k) = u_s(k) = 1$, then for all finitely generated field extensions~$L$ over $k$ of transcendence degree one, $u(L) = 2$ by Proposition~\ref{u-inv if m_s = u_s}, and $m(L) \geq 2$ by Lemma \ref{m at least 2 for fg trdeg 1}, so $m(L) = u(L) = 2$. If~$k$ is a finite field of odd characteristic, then $m_s(k) = u_s(k) = 2$, and for all finitely generated field extensions $L$ over $k$ of transcendence degree one, $m(L) = u(L) = 4$ \cite[Example~2.8]{m-inv}. We also note that the second condition above on $k$ does not follow from the first, as illustrated by~$\Q_p$ for $p$ an odd prime (see Example \ref{Q_p for m_s}).

We will now use the layer of a semi-global field $F$ over an $n$-local field to calculate $m(F)$. We will start with fully arboreal semi-global fields, and will need the following preliminary result about the set $\au(F)$ in the case that $n = 1$.
\begin{prop}
\label{au for 1-local}
Let $k$ be a field of characteristic $\ne 2$ such that
\begin{itemize}
	\item $m_s(k) = u_s(k) < \infty$, and
	
	\item $m(L) = u(L) = 2u_s(k)$ for all finitely generated field extensions $L$ over $k$ of transcendence degree one.
\end{itemize}
Let $T$ be a complete discrete valuation ring with residue field $k$, and let $\mathscr{X}$ be a regular connected projective $T$-curve with function field $F$ and reduction graph $\Gamma$. Then
\[
	\emph{AU}(F) = \begin{cases}
		\left\{2, 4u_s(k) \right\} &\text{ if $\Gamma$ is not a tree}, \\
		\left\{4u_s(k)\right\} &\text{ if $\Gamma$ is a tree}.
	\end{cases}
\]
\end{prop}
\begin{proof}
For the generic point $\eta$ of any irreducible component of the closed fiber of~$\mathscr{X}$, the field $\kappa(\eta)$ is a finitely generated transcendence degree one extension of $k$. So by our assumption on $k$, we have $m(\kappa(\eta)) = u(\kappa(\eta)) = 2u_s(k)$. The claim now follows immediately from Theorem \ref{au equals union}.
\end{proof}
\begin{examples} \label{1-local examples}
Let $K$ be a complete discretely valued field with residue field $k$ of characteristic $\ne 2$, and let $F$ be a one-variable function field over $K$. 
\begin{enumerate}[label=(\arabic*)]
	\item If $m_s(k) = u_s(k) = 1$ (e.g., $k$ is algebraically closed), then $\au(F) = \{4\}$ or $\{2, 4\}$.
	
	\item If $k$ is a finite field, then $\au(F) = \{8\}$ or $\{2, 8\}$.
\end{enumerate}
In both cases, $2 \in \au(F)$ if and only if the reduction graph a regular model of $F$ is not a tree.
\end{examples}

We now compute the $m$-invariant of a fully arboreal semi-global field.
\begin{prop}
\label{m-inv of fully arboreal}
Let $k$ be a field of characteristic $\ne 2$ such that
\begin{itemize}
	\item $m_s(k) = u_s(k) < \infty$, and
	
	\item $m(L) = u(L) = 2u_s(k)$ for all finitely generated field extensions $L$ over $k$ of transcendence degree one.
\end{itemize}
For any integer $n \geq 1$ let $K$ be an $n$-local field over $k$. 
If $F$ is a fully arboreal semi-global field over~$K$, then $m(F) = 2^{n+1}u_s(k)$.
\end{prop}
\begin{proof}
We proceed by induction on $n \geq 1$. First suppose $n = 1$. Since $F$ is fully arboreal, the reduction graph of any regular model of $F$ is a tree. So $\au(F) = \{4u_s(k)\}$ by Proposition \ref{au for 1-local}, hence $m(F) = 2^{1+1}u_s(k)$, proving the base case.

Now suppose for some integer $n \geq 1$ that $m(E) = 2^{n+1}u_s(k)$ for all fully arboreal semi-global fields $E$ over any $n$-local field over~$k$. Let $K$ be an $(n+1)$-local field over $k$, and let $F$ be a fully arboreal semi-global field over $K$. Since $F$ is fully arboreal, for the generic point $\eta$ of any irreducible component of the closed fiber of any regular model of $F$, the field $\kappa(\eta)$ is fully arboreal as well. By the induction hypothesis, $m(\kappa(\eta)) =$~$2^{n+1}u_s(k)$. This then implies, by Corollary \ref{m-inv of sg field is min of m-inv of residue fields}, that $m(F) = 2^{n+2}u_s(k)$. This completes the proof. 
\end{proof}
Next, we use the layer of a semi-global field $F$ that is not fully arboreal to compute $m(F)$.
\begin{prop}
\label{m-inv if layer = j}
Let $k$ be a field of characteristic $\ne 2$ such that 
\begin{itemize}
	\item $m_s(k) = u_s(k) < \infty$, and 
	
	\item $m(L) = u(L) = 2u_s(k)$ for all finitely generated field extensions $L$ over $k$ of transcendence degree one.
\end{itemize}
For any integer $n \geq 1$ let $K$ be an $n$-local field over $k$, and let $F$ be a one-variable function field over $K$. If $F$ is not fully arboreal, then $m(F) = 2^{\emph{layer}(F)}$.
\end{prop}
\begin{proof}
Since $F$ is not fully arboreal, we have $1 \leq \layer(F) \leq n$. We prove the claim by induction on $n \geq 1$. For the base case of $n = 1$, if $\layer(F) = 1$, then the reduction graph of any regular model of $F$ is not a tree, so $m(F) = 2 = 2^{\layer(F)}$ by Lemma~\ref{loop implies m=2}. This proves the base case.

Now suppose for some integer $n \geq 1$ that the claim is true for all $n$-local fields over $k$, and let $K$ be an $(n+1)$-local field over $k$. Let $F$ be a one-variable function field over $K$ with $1 \leq \layer(F) \leq n+1$, and let $\mathscr{X}$ be a regular model of $F$. Let $X$ be the closed fiber of $\mathscr{X}$ with irreducible components $X_1, \ldots, X_s$, and for $1 \leq i \leq s$ let $\eta_i$ be the generic point of $X_i$. As we saw in the base case, if $\layer(F) = 1$, then $m(F) = 2 = 2^{\layer(F)}$. If $1 < \layer(F) \leq$~$n+1$, then the reduction graph of~$\mathscr{X}$ is a tree and, as we saw in Remark \ref{layer in terms of components}, the set $I_0 =$~$\{i \in \{1, \ldots, s\} \mid$~$\kappa(\eta_i) \text{ not fully arboreal}\}$ is non-empty and $\layer(F) = \min_{i \in I_0} \{\layer(\kappa(\eta_i))\} + 1$. By the induction hypothesis, for each $i \in I_0$ we have $m(\kappa(\eta_i)) = 2^{\layer(\kappa(\eta_i))} \leq 2^n$. Moreover, if $i \in \{1, \ldots, s\} \setminus I_0$, then $\kappa(\eta_i)$ is fully arboreal and $m(\kappa(\eta_i)) = 2^{n+1}u_s(k)$ by Proposition~\ref{m-inv of fully arboreal}. So by Corollary \ref{m-inv of sg field is min of m-inv of residue fields} we have
\[
	m(F) = \min_{i = 1, \ldots, s} \{2m(\kappa(\eta_i))\} = \min_{i \in I_0} \{2m(\kappa(\eta_i))\} = \min_{i \in I_0} \left\{2^{\layer(\kappa(\eta_i)) + 1}\right\} = 2^{\layer(F)}.
\]
This completes the proof. 
\end{proof}

\begin{cor}
\label{possible m-invariants}
Let $k$ be a field of characteristic $\ne 2$ such that 
\begin{itemize}
	\item $m_s(k) = u_s(k) < \infty$, and
	
	\item $m(L) = u(L) = 2u_s(k)$ for all finitely generated field extensions $L$ over $k$ of transcendence degree one.
\end{itemize}
For any integer $n \geq 1$ let $K$ be an $n$-local field over $k$, and let $F$ be a one-variable function field over $K$. Then there is some integer $r \geq 0$ such that
\[
	m(F) \in \left\{2^j \mid j = 1, 2, \ldots, n, n + r + 1\right\}.
\]
Moreover, for each $j \in \{1, 2, \ldots, n, n+r+1\}$ there is a one-variable function field $F_j$ over $K$ with $m(F_j) = 2^j$.
\end{cor}
\begin{proof}
Since $m_s(k) = u_s(k) < \infty$, Corollary~\ref{if strong m = strong u} implies that $m_s(k) = u_s(k) = 2^r$ for some integer $r \geq 0$. The first claim of the corollary then follows from Propositions \ref{m-inv of fully arboreal} and \ref{m-inv if layer = j}.

Now consider the second claim of the corollary. By Proposition \ref{curves with layer j}, for each $j = 1, \ldots, n$ there is a one-variable function field $F_j$ over $K$ with $\layer(F_j) = j$. Thus $m(F_j) = 2^j$ by Proposition \ref{m-inv if layer = j}. Next, if $F_{n+r+1} = K(x)$, then $F_{n+r+1}$ is fully arboreal (see Example \ref{fully arboreal example}), so $m(F_{n+r+1}) = 2^{n+r+1}$ by Proposition \ref{m-inv of fully arboreal}.
\end{proof}

\begin{examples}
\label{m-invariants over n-local fields}
For any integer $n \geq 1$ let $K$ be an $n$-local field over a field $k$ of characteristic $\ne 2$, and let $F$ be a one-variable function field over $K$. 
\begin{enumerate}[label=(\arabic*)]
	\item If $m_s(k) = u_s(k) = 1$ (e.g., $k$ is algebraically closed), then $m(F) \in \left\{2^j \mid j = 1, 2, \ldots, n, n+1\right\}$.
	
	\item If $k$ is a finite field, then $m(F) \in \left\{2^j \mid j = 1, 2, \ldots, n, n+2 \right\}$.
\end{enumerate}
\end{examples}

\begin{remark}
For any integer $n \geq 2$ let $K$ be an $n$-local field over a field $k$ of characteristic $\ne 2$ such that $m_s(k) = u_s(k) < \infty$, and let $F$ be a one-variable function field over $K$. As we have seen, if $\eta$ is the generic point of an irreducible component of the closed fiber of any regular model of $F$, then the field $\kappa(\eta)$ is itself a semi-global field over an $(n-1)$-local field over $k$. As a result, we can inductively use Theorem \ref{au equals union} to calculate $\au(F)$, as illustrated by the following examples.
\end{remark}
\begin{examples}
\begin{enumerate}[label=(\arabic*)]
	\item Let $k$ be a finite field of odd characteristic. If $F_1$ is a one-variable function field over a 1-local field $K_1$ over $k$, then $\au(F_1) = \{8\} \text{ or } \{2, 8\}$ (see Examples \ref{1-local examples}(2)). Now, if $F_2$ is a one-variable function field over a 2-local field $K_2$ over $k$, then by the previous sentence and Theorem~\ref{au equals union} we conclude $\au(F_2) = \{16\}, \{2, 16\}, \{4, 10, 16\}, \text{ or } \{2, 4, 10, 16\}$. Furthermore, we can use the geometry associated to $F_2$ to distinguish among these four cases. For instance, if $F_2$ has a regular model whose reduction graph is a tree and whose closed fiber has an irreducible component with generic point $\eta$ such that the reduction graph of a regular model of $\kappa(\eta)$ is not a tree, then $\au(F_2) = \{4, 10, 16\}$.
	
	\item If $F_3$ is a one-variable function field over a 3-local field $K_3$ over an algebraically closed field~$k$ of characteristic $\ne 2$, then there are ten possibilities for $\au(F_3)$:
	\begin{align*}
		&\{16\}, \{4, 10, 16\}, \{8, 10, 12, 14, 16\}, \{4, 8, 10, 12, 14, 16\}, \{4, 6, 8, 10, 12, 14, 16\}, \{2, 16\}, \\
		&\{2, 4, 10, 16\}, \{2, 8, 10, 12, 14, 16\}, \{2, 4, 8, 10, 12, 14, 16\}, \text{ and } \{2, 4, 6, 8, 10, 12, 14, 16\}.
	\end{align*}
Similar geometric considerations to those seen in the previous example can be used to distinguish among these ten cases.
\end{enumerate}
\end{examples}

\begin{ack} The author would like to thank David Harbater, Julia Hartmann, Daniel Krashen, and Florian Pop for inspiring discussions and comments regarding the material of this article. This paper is based on part of the author's Ph.D. thesis, completed under the supervision of David Harbater at the University of Pennsylvania.
\end{ack}

\providecommand{\bysame}{\leavevmode\hbox to3em{\hrulefill}\thinspace}
\providecommand{\href}[2]{#2}

\medskip

\noindent{\bf Author Information:}\\

\noindent Connor Cassady\\
Department of Mathematics, The Ohio State University, Columbus, OH 43210-1174, USA\\
email: cassady.82@osu.edu


\begin{thebibliography}{KMRT98}

\bibitem[Abh69]{abh}
S.S.\ Abhyankar,
\textit{Resolution of singularities of algebraic surfaces},
Algebraic Geometry (Internat. Colloq., Tata Inst. Fund. Res., Bombay, 1968), pp. 1--11, Oxford University Press, London, 1969.

\bibitem[AB59]{auslander}
M.\ Auslander, D.A.\ Buchsbaum,
\textit{Unique factorization in regular local rings},
Proc. Nat. Acad. Sci. U.S.A. \textbf{45} (1959), 733--734.

\bibitem[BL14]{becher-leep}
K.J.\ Becher, D.\ Leep,
\textit{The Kaplansky radical of a quadratic field extension},
J. Pure Appl. Algebra \textbf{218} (2014), no. 9, 1577--1582.

\bibitem[Bha00]{bhargava}
M.\ Bhargava,
\textit{On the Conway-Schneeberger fifteen theorem},
Quadratic forms and their applications (Dublin, 1999), 27--37, Contemp. Math. \textbf{272}, Amer. Math. Soc., Providence, RI, 2000.

\bibitem[BH05]{bh05}
M.\ Bhargava, J.\ Hanke,
\textit{Universal quadratic forms and the 290-theorem}, Preprint, 2005.

\bibitem[BK18]{bk18}
V.\ Blomer, V.\ Kala,
\textit{On the rank of universal quadratic forms over real quadratic fields},
Doc. Math. \textbf{23} (2018), 15--34.

\bibitem[Cas23]{cas23}
C.\ Cassady,
\textit{Hasse principles for quadratic forms over function fields},
J. Algebra \textbf{628} (2023), 613--633.

\bibitem[CPS12]{cps12}
J.-L.\ Colliot-Th\'el\`ene, R. Parimala, V. Suresh,
\textit{Patching and local-global principles for homogeneous spaces over function fields of $p$-adic curves},
Comment. Math. Helv. \textbf{87} (2012), no. 4, 1011--1033.

\bibitem[Con00]{conway}
J.H.\ Conway,
\textit{Universal quadratic forms and the fifteen theorem}, Quadratic forms and their applications (Dublin, 1999), 23--26, Contemp. Math. \textbf{272}, Amer. Math. Soc., Providence, RI, 2000.

\bibitem[DM69]{deligne}
P.\ Deligne, D.\ Mumford,
\textit{The irreducibility of the space of curves of given genus},
Inst. Hautes \'{E}tudes Sci. Publ. Math. \textbf{36} (1969), 75--109. 

\bibitem[GVG92]{m-inv}
N.\ Gesqui\`{e}re, J.\ Van Geel,
\textit{Note on universal quadratic forms},
Bull. Soc. Math. Belg. S\'{e}r. B \textbf{44} (1992), no. 2, 193--205.

\bibitem[GW10]{gw}
U.\ G\"{o}rtz, T. Wedhorn.
\textit{Algebraic geometry I. Schemes with examples and exercises},
Advanced Lectures in Mathematics, Vieweg + Teubner, Wiesbaden, 2010.

\bibitem[Gup21]{gup}
P.\ Gupta,
\textit{Four-dimensional quadratic forms over $\C((t))(x)$},
Arch. Math. (Basel) \textbf{117} (2021), no. 4, 369--374.

\bibitem[HH10]{hh10}
D.\ Harbater, J.\ Hartmann,
\textit{Patching over fields},
Israel J. Math. \textbf{176} (2010), 61--107.

\bibitem[HHK09]{hhk09}
D.\ Harbater, J.\ Hartmann, D.\ Krashen,
\textit{Applications of patching to quadratic forms and central simple algebras},
Invent. Math. \textbf{178} (2009), no. 2, 231--263.

\bibitem[HHK13]{weier}
D.\ Harbater, J.\ Hartmann, D.\ Krashen,
\textit{Weierstrass preparation and algebraic invariants},
Math. Ann. \textbf{356} (2013), no. 4, 1405--1424.

\bibitem[HHK15a]{hhk15}
D.\ Harbater, J.\ Hartmann, D.\ Krashen,
\textit{Local-global principles for torsors over arithmetic curves},
Amer. J. Math. \textbf{137} (2015), no. 6, 1559--1612.

\bibitem[HHK15b]{refine}
D.\ Harbater, J.\ Hartmann, D.\ Krashen,
\textit{Refinements to patching and applications to field invariants},
Int. Math. Res. Not. IMRN (2015), no. 20, 10399--10450.

\bibitem[HKP21]{hkp}
D.\ Harbater, D.\ Krashen, A.\ Pirutka,
\textit{Local-global principles for curves over semi-global fields},
Bull. Lond. Math. Soc. \textbf{53} (2021), no. 1, 177--193.

\bibitem[Hof94]{hoff}
D.W.\ Hoffmann,
\textit{On 6-dimensional quadratic forms isotropic over the function field of a quadric},
Comm. Algebra \textbf{22} (1994), no. 6, 1999--2014.

\bibitem[Hu12]{hu12}
Y.\ Hu,
\textit{Local-global principle for quadratic forms over fraction fields of two-dimensional Henselian domains},
Ann. Inst. Fourier (Grenoble) \textbf{62} (2012), no. 6, 2131--2143.

\bibitem[KY21]{ky21}
V.\ Kala, P.\ Yatsyna,
\textit{Lifting problem for universal quadratic forms},
Adv. Math. \textbf{377} (2021), Paper No. 107497, 24 pp.

\bibitem[Kap69]{kap}
I.\ Kaplansky,
\textit{Fr\"{o}hlich's local quadratic forms},
J. Reine. Angew. Math. \textbf{239-240} (1969), 74--77.

\bibitem[Lam05]{lam}
T.Y.\ Lam,
\textit{Introduction to quadratic forms over fields},
Graduate Studies in Mathematics \textbf{67},
Amer. Math. Soc., Providence, RI, 2005. 

\bibitem[Lip75]{lip}
J.\ Lipman,
\textit{Introduction to resolution of singularities},
Algebraic geometry (Proc. Sympos. Pure Math., Vol. 29, Humboldt State Univ., Arcata, Calif., 1974), pp. 187--230, Amer. Math. Soc., Providence, RI, 1975.

\bibitem[Liu02]{liu}
Q.\ Liu,
\textit{Algebraic geometry and arithmetic curves}, 
Oxford Graduate Texts in Mathematics \textbf{6}, Oxford University Press, Oxford, 2002.

\bibitem[Mil20]{milne}
J.S.\ Milne,
\textit{Algebraic Number Theory} (v3.08),
2020, Available at www.jmilne.org/math/.

\bibitem[Rou14]{rouse}
J.\ Rouse,
\textit{Quadratic forms representing all odd positive integers},
Amer. J. Math. \textbf{136} (2014), no. 6, 1693--1745.

\bibitem[Sta23]{stacks-project}
The Stacks project authors,
The Stacks project, \url{https://stacks.math.columbia.edu}, 2023.

\bibitem[ZS75]{zar-sam}
O.\ Zariski, P.\ Samuel,
\textit{Commutative algebra. Vol. II.},
Graduate Texts in Mathematics \textbf{29}, Springer-Verlag, New York-Heidelberg, 1975.

\end{thebibliography}
\end{document}